\documentclass[11pt]{amsart}
\usepackage{amssymb,mathrsfs,graphicx,extpfeil}
\usepackage{epsfig}
\usepackage{indentfirst, latexsym, amssymb, enumerate,amsmath,graphicx}
\usepackage{float}
\usepackage{colortbl}
\usepackage{epsfig,subfigure}
\usepackage{mathtools}
\usepackage{amsmath}
\title[Discrete Kuramoto model]{Emergent behaviors of the discrete-time Kuramoto model for generic initial configuration}

\author[Zhang]{Xiongtao Zhang}
\address[Xiongtao Zhang]{\newline Center for Mathematical sciences, \newline Huazhong University of Science and Technology, Wuhan, China}
\email{xtzhang@hust.edu.cn}

\author[Zhu]{Tingting Zhu}
\address[Tingting Zhu]{\newline School of Mathematics and Statistics, \newline Huazhong University of Science and Technology, Wuhan, China }
\email{ttzhu@hust.edu.cn}

\newtheorem{theorem}{Theorem}[section]
\newtheorem{lemma}{Lemma}[section]

\newtheorem{remark}{Remark}[section]

\newcommand{\bbr}{\mathbb R}

\newcommand{\bbz}{\mathbb Z}

\def\charf {\mbox{{\text 1}\kern-.30em {\text l}}}
\begin{document}
%%%%%%%%%%%%%%%%

\date{\today}

\subjclass{39A10, 39A12, 34D05, 68M10.} 
\keywords{Discrete-time gradient flow, Kuramoto model, discrete-time dynamics, generic initial data, uniform convergence}

\thanks{The work of X. Zhang is supported by the National Natural Science Foundation of China (Grant No. 11801194).}

%\thanks{$^*$ Corresponding author: }
%\thanks{\textbf{} }

\begin{abstract}
 In this paper, we will study the emergent dynamics of the discrete Kuramoto model for generic initial data. This is an extension of the previous work \cite{H-K-K-Z}, in which the initial configurations are supposed to be within a half circle. More precisely, we will provide the theory of discrete gradient flow which can be applied to general Euler iteration scheme. Therefore, as a direct application, we conclude the emergence of synchronization of discrete Kuramoto model. Moreover, we obtain for small mesh size that, the synchronization will occur exponentially fast for initial data in $\mathcal{A}_1$ (see definition in \eqref{D1}).  
\end{abstract}
\maketitle \centerline{\date}

%\tableofcontents
\section{Introduction}\label{sec:1}
\vspace{0.5cm}	
	Collective dynamics of complex systems exist all around the world, in which self-propelled individuals organize themselves into a particular motion through simple rules. For instance, the aggregation of bacteria, flocking of birds, swarming of fish, and even the motion of galaxy can be considered as various types of complex systems \cite{B-H, C-F-T-V, C-K-M-T, D-M2, H-K-P-Z2, J-K, P-L-S-G-P, P-R-K, T-T, T-B, V-C-B-C-S, Wi}. To model such collective dynamics, several phenomenological models were proposed and have been studied analytically and numerically \cite{B-C-C, C-C-R, C-D-P, C-H, C-H-L, C-M,  C-S2, D-M1, D-M2, D-M3, D-Y, D-F-T, F-H-J, H-L-R-S, H-Liu, H-K-P-Z3, H-P-Z, H-T, LH, LX, S}. Recently, due to the relation with engineering applications such as control of robots, sensor networks and formation of unmanned aerial vehicle etc., the collective behaviors in complex systems has been extensively studied \cite{L-P-L-S, M-T2, M-T1, P-L-S-G-P, P-E-G}. In this paper, we will consider the well-known Kuramoto model describing the motion of oscillators on the unit circle $S^1$. The dynamics of Kuramoto model are given by the following ordinary differential equations:
\begin{equation}\label{A-1}
\dot{\theta}_i=\Omega_i+\frac{K}{N}\sum_{j=1}^N \sin(\theta_j-\theta_i).
\end{equation}
where $\theta_i \in \bbr$ and $\Omega_i$ are the phase and a natural frequency of $i$-th oscillator, respectively.
The Kuramoto model \eqref{A-1} has been extensively studied in many papers, to name a few, synchronization and stability \cite{A-B-V-R-S, C-H-J-K, C-S, D-X, D-B, H-H-K, H-K-K-Z, H-K-R, H-K-P-Z2, H-L-X, J-M-B, St}, mean-field limit and corresponding neutral stability in kinetic Kuramoto model \cite{B-C-M, H-K-P-Z, S-M, S-M2}, Landau damping around incoherent state and partial phase locked state \cite{D-F-G, F-G-G} and the sensitivity analysis \cite{H-J-J}.\newline 

However, the continuous-time model is ideal one and every time when we perform numerical simulations for the continuous-time model, we can only use discrete-time model. Thus, it is necessary to study the corresponding discrete-time model. To discretize the continuous-time model \eqref{A-1}, we follow \cite{H-K-K-Z} to choose the forward Euler method. More precisely, for fixed time-step size $h := \Delta t > 0$, let $\theta_i^h(n)$ be the phase for the $i$-th agent evaluated at the $n$-th step. Then, for the original Kuramoto model, the dynamics of $\theta_i^h(n)$ is governed by the following discrete system: 
\begin{equation} \label{A-2}
\begin{cases}
\displaystyle \theta_i^h(n+1) = \theta_i^h(n)+ h \Omega_i+\frac{hK}{N}\sum_{j=1}^N\sin(\theta_j^h(n)-\theta_i^h(n)), \quad n = 0, 1, \cdots, \quad 1 \leq i \leq N, \\
\displaystyle \theta_i^h(0) = \theta_i^0, \quad  \sum_{i=1}^{N} \theta_i^0 = 0.
\end{cases}
\end{equation}\newline
The discrete-time model has an advantage that, the time asymptotical behavior of a non-all-to-all and non-symmetric model can be obtained by studying an corresponding all-to-all symmetric system after several steps of iteration, while this kind of analysis cannot be directly applied to the continuous model \cite{C-K-M-T, CD1, LH, LX, S}. Therefore, in order to understand the large time behavior of Kuramoto oscillators, it's very important to study the discrete Kuramoto model under symmetric topology, i.e. \eqref{A-2}. However, there are two main difficulty to deal with the discrete Kuramoto model \eqref{A-2}. First, we cannot follow the analysis on continuous model to construct the differential equation of diameter, since the existence of error along the iteration. Thus we have to carefully estimate the higher order error and yield the time asymptotical behavior. Secondly, since the half circle is an invariant set for large coupling strength and, the authors  in \cite{H-K-K-Z} studied the identical model in half circle and non-identical model in a quarter, respectively. However, the half circle is no more an invariant set  for generic initial data, thus we have to develop some new technique to overcome this kind of difficulty. In \cite{H-K-R}, the authors applied the properties of gradient flow to prove for continuous-time model that, frequency synchronization will occur for generic initial data, provided the coupling strength is sufficiently large. Therefore, it is very natural to ask wether similar results can be rigorously proved for discrete-time model. More precisely, we address the following questions in the present paper:\newline

 \begin{itemize}
\item Question A: (Discrete gradient flow) Since the right hand side of the iteration scheme \eqref{A-2} preserves the gradient structure, is it possible to construct the equilibrium state and the corresponding convergence as  in continuous time model?
\vspace{0.2cm}

\item Question B: (Synchronization) Can we combine the discrete gradient flow structure and higher order estimates on iteration errors to verify the emergence of the synchronization for generic initial data, and can we obtain the asymptotical convergence rate?
\end{itemize}
\vspace{0.5cm}

Our main results in the present paper are three-fold. First, we will estimate the higher order error in the general discrete gradient flow and prove the asymptotical convergence of the discrete gradient flow. Different from the continuous version, we have to construct a convex combination of two adjacent steps to control the error from the iteration, and thus yield the desired result  (see Theorem \ref{T3.1}). Secondly, based on the well approximation between continuous-time model and discrete-time model, we will study the identical discrete-time Kuramoto model and show, for almost all initial data, the exponential convergence to either phase synchronization state or bipolar state. Moreover, we will construct all possible equilibrium state for identical discrete Kuramoto model, which is the same as the continuous model (see Theorem \ref{T4.1}). Third, we will follow the similar idea in \cite{H-K-R} and apply the discrete gradient flow theory to prove the emergence of synchronization for non-identical discrete Kuramoto model (see Theorem \ref{final}).\newline 

The rest of the present paper is organized as follows. In Section \ref{sec:2}, we will review some well known preliminary results such as the asymptotical behavior of the continuous Kuramoto model, the total error estimates of the Euler scheme, the $\L$ojasiewicz inequality and the corresponding convergence result of gradient flow, etc. In Section \ref{sec:3}, we will provide the theory of discrete gradient flow and prove it by higher order of error estimates. In Section \ref{sec:4}, we will construct all the possible equilibrium states of the identical discrete Kuramoto model and show the exponential convergence of almost all initial data to one of the equilibrium states. In Section \ref{sec:5}, based on the results in the previous sections, we will prove the emergence of the synchronization of discrete non-identical Kuramoto model for almost all initial data, provided the coupling strength is suffciently large. Finally, Section \ref{sec:6} will be devoted to a brief summary.\newline

\section{Preliminaries}\label{sec:2}
\setcounter{equation}{0}
\vspace{0.5cm}
In this section, we will review some previous results for the continuous Kuramoto model \eqref{A-1} in both identical and nonidentical case. Then, we will introduce some well known lemmas such as total error for discrete Euler scheme and asymptotical behavior of gradient flow, which will be mainly used in the later sections. 

\subsection{Identical Kuramoto model}
In this part, we will review some previous results for identical Kuramoto model. First, we introduce the definition of order parameter which has been widely used in the study of Kuramoto model. More precisely, for a configuration $\Theta = \Theta(t)=(\theta_1(t), \theta_2(t), \cdots, \theta_N(t))$ governed by \eqref{A-1}, the Kuramoto order parameters are defined by the following relation:
\begin{equation}\label{B1}
 re^{i\phi} := \frac{1}{N} \sum_{k=1}^{N} e^{i\theta_k},\quad \theta_j(0) = \theta_j^0, \quad r_0 := \left| \frac{1}{N} \sum_{j=1}^{N} e^{i\theta_j^0} \right|.
 \end{equation}
For identical oscillators, without loss of generality, we may assume $\Omega_i = 0$ for all $i$ due to the conservation of the mean natural frequency. Then, we first recall the result in \cite{B-C-M, H-K-R} which provided the possible asymptotic states for identical Kuramoto oscillators.
\begin{lemma}	
\label{L2.1}
\cite{B-C-M}
Let $\Theta = (\theta_1, \ldots,\theta_N)$ be a solution to the identical Kuramoto model \eqref{A-1}
with initial phases $\Theta_0$ satisfying
$\frac{1}{N} \sum_{j=1}^N \theta_j^0 = 0$
and $r_0 > 0$, where $r_0$ is defined in \eqref{B1}. Then, we have
\[ \lim\limits_{t \to \infty} |\theta_j(t) - \phi(t)| = 0\ \  \text{or}\ \  \pi, \quad \text{for all} \ j =1, \dots, N.\]
\end{lemma}
\begin{lemma}
\label{L2.2}
\cite{B-C-M, H-K-R}
Let $\Theta = (\theta_1, \ldots,\theta_N)$ be a solution to the identical continuous Kuramoto model \eqref{A-1}
with natural frequency $\Omega_i$ and initial configuration $\Theta_0$ satisfying
\begin{equation}\label{B2}
 \Omega_i=0,\quad \frac{1}{N} \sum_{j=1}^N \theta_j^0 = 0,\quad  \theta_i^0 \neq \theta_j^0,\quad  i \ne j,\quad r_0 > 0,
 \end{equation}
%\quad \theta_j^0 \in [-\pi, \pi), \quad 1 \le i,j \le N, 
where $r_0$ is defined in \eqref{B1}. Moreover, we define the synchronization set $\mathcal{I}_s$ and bipolar set $\mathcal{I}_b$ as follows,
\[ \mathcal{I}_s := \{ j: \lim_{t \to \infty} |\theta_j(t) - \phi(t)| = 0 \}, \quad \mathcal{I}_b := \{ j: \lim_{t \to \infty} |\theta_j(t) - \phi(t)| = \pi \}.\]
Then, we have $|\mathcal{I}_b| \leq 1$, where $|A|$ is the cardinality of the set $A$.
\end{lemma}
\begin{remark}\label{R2.1}
In the proof of Lemma \ref{L2.1} in \cite{B-C-M}, the authors constructed the time asymptotical limits of $\phi(t)$ and $\theta_j(t)$, $(j =1, \ldots, N)$, which are given as follows,
\[\lim_{t \to + \infty} \phi(t) = -\frac{1}{N} \sum_{j=1}^N k_j \pi := \phi^*, \quad \lim_{t \to +\infty} \theta_j(t) = k_j \pi + \phi^*, \quad k_j \in \bbz.\]
Moreover, the condition $\theta_i^0 \neq \theta_j^0$ means the initial data are chosen in the set 
\[\bbr^N\setminus\Big(\cup_{i\neq j}\{\Theta\ :\theta_i=\theta_j\}\cup\{\Theta\ : r_0=0\}\Big).\] 
As the sets $\{\Theta\ :\theta_i=\theta_j\}$ and $\{\Theta\ : r_0=0\}$ are all lower dimensional manifold in $\bbr^N$, we immediately conclude the set $\cup_{i\neq j}\{\Theta\ :\theta_i=\theta_j\}$ and $\{\Theta\ : r_0=0\}$ are measure zero in $\bbr^N$. Therefore, Lemma \ref{L2.2} holds for almost all initial data $\Theta_0\in\bbr^N$.
%where the value of $k_j$ depends on the corresponding $j$-th oscillator. 
\end{remark}

\subsection{Nonidentical Kuramoto model}
In this part, we will review some results for nonidentical Kuramoto oscillators. Actually, there are many literatures related to the nonidentical Kuramoto model, but we will mainly focus on the asymptotical properties of oscillators on half circle and whole circle, respectively. 
\begin{lemma}\label{L2.3}    
\cite{C-H-J-K}
Let $\Theta = \Theta(t)$ be the global smooth solution to the continuous Kuramoto model \eqref{A-1} subject to initial data $\theta_i(0) = \theta_i^0$ and satisfying
\[0 < D(\Theta_0) < \pi, \qquad D(\Omega) > 0, \qquad K > K_e,\]
where $K_e = \frac{D(\Omega)}{\sin D(\Theta_0)}$. Then there exists $t_0 > 0$ such that
\[ D(\Theta(t)) \le \arcsin(\sin D(\Theta_0)) \text{\quad for } t \ge t_0.\]
%	   where $D^\infty$ is the dual angle of initial phase diameter $D(\Theta_0)$ satisfying
%	   \[ D^\infty \in (0, \frac{\pi}{2}), \qquad \sin D^\infty = \sin D(\Theta_0).\]
\end{lemma}

\begin{lemma}\label{L2.4}
\cite{H-K-R} Suppose that the initial configuration $\Theta_0$ and the natural frequencies $\Omega_i$ satisfy the conditions below
\begin{equation}\label{multiple}
\left\{
\begin{aligned}
&~\frac{1}{N} \sum_{j=1}^N \Omega_j = 0, \quad \frac{1}{N} \sum_{j=1}^N \theta_j^0 = 0, \quad \theta_j^0 \in [-\pi, \pi), \quad 1 \le j \le N,\\
&~r_0 > 0, \quad \theta_j^0 \ne \theta_k^0, \quad 1 \le j \ne k \le N, \quad \max_{1 \le j \le N} |\Omega_j| < \infty.
\end{aligned}
\right.
\end{equation}
Then there exists a positive constant $K_\infty > 0$ such that, for sufficient large coupling strength $K>K_{\infty}$, there exists an asymptotical phase-locked state $\Theta^\infty$ satisfying
\[\lim_{t \to \infty} || \Theta(t) - \Theta^\infty||_\infty = 0,\]
where the norm $||\cdot||_\infty$ is the standard $l^\infty$-norm in $\bbr^N$.
\end{lemma}

\begin{remark}\label{R2.2}
Actually, in the original papers \cite{C-H-J-K, H-K-R}, the authors constructed more detailed structures in the proof of the Lemma \ref{L2.3} and Lemma \ref{L2.4}.
\begin{enumerate}
\item The authors in \cite{C-H-J-K} actually constructed a positive constant $D^\infty\in(0,\frac{\pi}{2})$, such that $D(\Theta(t)) < D^\infty$ for $t \ge t_0$.
\item For initial data satisfying \eqref{multiple} and sufficiently large coupling strength $K\geq K_{\infty}$, the authors in \cite{H-K-R} actually constructed positive constants $N_0$, $l$ and a time $T_*$ such that,
\begin{equation*}
N_0 \in \bbz^+ \cap \left(\frac{N}{2}, N\right], \quad l \in \left(0, 2 \arccos \frac{N-N_0}{N_0}\right),\quad \max_{1 \le j,k \le N_0} |\theta_j(T_*) - \theta_k(T_*)| < l.
\end{equation*}
\end{enumerate}
\end{remark}

\subsection{Preliminary lemmas}
In this part, we will provide some well known classical results which we will frequently used in the later sections. We first review the classical numerical analysis for the Euler scheme. Consider Cauchy problems for the first-order autonomous ODE system and its corresponding discretized system obtained by the one-step forward Euler scheme with the same initial data: for $T \in (0, \infty]$,
\begin{equation} \label{F-1}
\begin{cases}
\frac{dy}{dt} =f(y),\quad 0\leq t\leq T, \\ 
y(0)=y_0, 
\end{cases}
\quad \mbox{and} \qquad 
\begin{cases}
y_{n+1}=y_n+hf(y_n), \quad n \geq 0, \\
y_0=y(0).
\end{cases}
\end{equation}
Then, a standard convergence results from the discretized system to the continuous system in \eqref{F-1} can be summarized in the following proposition. We first introduce  ``{\it truncation error}" ${\mathcal E}^h _{1}(n)$ and ``{\it global error}" ${\mathcal E}^h_{2}(n)$" as follows. 
\[  {\mathcal E}^h _{1}(n):= \Big \| \frac{dy}{dt} \Big|_{t = nh}  - \frac{y_{n+1}-y_n}{h} \Big \|, \qquad {\mathcal E}^h_{2}(n):=\|y(nh)-y_n \|. \]
\begin{lemma}\label{L2.5}
\emph{\cite{S-Ma}}
Let $T, R \in (0, \infty)$ be positive constants, and suppose that the continuous and discrete system satisfy
\begin{enumerate}
\item
The forcing function $f$ is Lipschitz continuous on the open set ${\mathcal D}$ with Lipschitz constant ${\mathcal L}_f$: 
\[ {\mathcal D} :=\{(t,y)\ :\ 0\leq t\leq T,\quad \|y-y_0\|\leq R\}, \quad \sup_{y_1 \neq y_2, y_1, y_2 \in {\mathcal D}} \frac{|f(y_2) - f(y_1)|}{|y_2 - y_1|} = {\mathcal L}_f < \infty. \] 
\item
The discrete values $y_n$ obtained by the discrete system in \eqref{F-1} satisfy
\[  \|y_n-y_0\|\leq R, \quad \mbox{for all}~n = 0, 1, \cdots, \Big[\frac{T}{h} \Big].  \]
\end{enumerate}
Then, we have the following consistency and convergence results.
\begin{enumerate}
\item (Consistency): Maximal truncation error tends to zero, as time-step tends to zero:
\[\lim\limits_{h\rightarrow 0}\max\limits_{0 \leq n \leq [T/h]}  {\mathcal E}^h _{1}(n) =0. \]
\item (Convergence): the global error can be controlled by the truncation error, more precisely, 
\begin{equation} \label{F-1-1}
 {\mathcal E}^h_{2}(n)  \leq \frac{ \max\limits_{0 \leq n \leq [T/h]}  {\mathcal E}^h _{1}(n)  )}{{\mathcal L}_{f}}\left(e^{{\mathcal L}_{f}nh}-1\right), \quad 0 \leq n \leq [T/h]. 
 \end{equation}
\end{enumerate}
\end{lemma}
\begin{remark} \label{R2.3}
Note that the term $e^{{\mathcal L}_{f}nh}-1$ in the right-hand side of \eqref{F-1-1} grows exponentially. Thus, the convergence result in Lemma \ref{L2.5} is valid only for a finite-time interval. 
\end{remark}
\noindent Next, we introduce some properties of the gradient flow system and a simple inequality of concave functions. The gradient flow structure is quite important to the emergence of synchronization in Kuramoto model, while the sub-additive property of concave functions will be frequently used in the later sections. 
\begin{lemma}\label{L-inequality}
\cite{D-X} ($\L$ojasiewicz inequality) Suppose that $f:D\subseteq\bbr^n\rightarrow \bbr$ is analytic in the open set $D$. Let $\bar{x}$ be a critical point of $f$, i.e., $\nabla f(\bar{x})=0$. Then there exist $r>0$, $c>0$, and $\eta\in[\frac{1}{2},1)$ such that 
\[\|\nabla f(x)\|\geq c|f(x)-f(\bar{x})|^{\eta},\quad \forall x\in B(\bar{x},r).\]
\end{lemma}

\begin{lemma}\label{c-gradient} 
\cite{D-X} Suppose $f(x)$ is an analytic function. Let $x(t)$ be uniformly bounded and follow a gradient flow with $f(x)$ to be the potential i.e. $\dot{x}=-\nabla_xf(x)$. Then $x(t)$ converges to a limit $x^{\infty}$. 
\end{lemma}

\begin{lemma}\label{sub-additive}
Let $f(x)$ be a concave function defined on $[0, +\infty)$ and $f(0) \geq 0$, then f is sub-additive on $[0,+\infty)$ i.e.
\[f(a)+f(b) \geq f(a+b),\quad a,b \in [0,+\infty).\]
\end{lemma}
\begin{proof}
Due to the fact that $f(x)$ is concave and $f(0)\geq 0$, we immediately obtain the following inequality
\[ f(tx) = f(tx+(1-t) \cdot 0) \geq tf(x) + (1-t)f(0) \geq tf(x), \quad x \in (0,+\infty),\quad 0\leq t\leq 1.\]
Therefore, for $a,b \in [0, +\infty)$, we have
\begin{align*}
\begin{aligned}
f(a) + f(b) = f\left(  \frac{a(a+b)}{a+b} \right) + f\left( \frac{(a+b)b}{a+b} \right) \geq \frac{a}{a+b} f(a+b) + \frac{b}{a+b}f(a+b)= f(a+b).
\end{aligned}
\end{align*}
\end{proof}

\section{Discrete gradient flow}\label{sec:3}
\setcounter{equation}{0}
\vspace{0.5cm}
In this section, we will introduce an asymptotical stability for the discrete gradient flow. It is well known that the gradient flow structure of dynamical system is very important because it will generally lead to the stability of the system. For instance, the author applied Lemma \ref{L-inequality} and Lemma \ref{c-gradient} in \cite{D-X} to show the emergence of synchronization of Kuramoto oscillators. However, the Lemma \ref{c-gradient} cannot be directly applied to the discrete model. Therefore, it is very important to establish a discrete version of such kind of stability theory from the discrete gradient flow structure.
\begin{theorem}\label{T3.1} 
Suppose $f(x)$ is an analytic function in a convex compact domain $D$. Let $x(n)$ be uniformly bounded in the compact domain $D$ for any $n$, and follow a gradient flow in discrete sense with $f(x)$ to be the potential i.e. 
\begin{equation}\label{C1}
x(n+1)-x(n)=-\nabla_xf(x(n))h,
\end{equation}
where $h$ is the mesh size. Then for sufficiently small $h$, there exists a $x^{\infty}$ such that 
\[\lim_{n\rightarrow+\infty}x(n)=x^{\infty},\quad \nabla_xf(x^{\infty})=0.\]
\end{theorem}
\begin{proof}
$\bullet$ (Step 1) In the first step, we construct the limit $x^{\infty}$. Since $x(n)$ is uniformly bounded in $D$, we immediately obtain that there exists a subsequence $x(n_k)$ and $x^{\infty}$ such that 
\begin{equation}\label{C2}
\lim\limits_{k\rightarrow+\infty}x(n_k)=x^{\infty}.
\end{equation} 
In the following, we will prove this subsequence limit actually is the limit of the whole sequence.\newline

\noindent$\bullet$ (Step 2) In the second step, we show that $x^{\infty}$ is a critical point. As $x(n)$ is uniformly bounded and $f(x)$ is analytic, the second order derivatives of $f(x)$ can reach the maximum and minimum value. More precisely, there exists a positive constant $C$ such that 
\begin{equation}\label{C3}
|\partial_{x_i}\partial_{x_j}f(x)|\leq C,\quad x\in D,\quad C=\max_{i,j}\max_{x\in D}|\partial_{x_i}\partial_{x_j}f(x)|.
\end{equation}
Then let $H(x)$ be the hessian matrix at $x$, we apply the Taylor expansion to imply that there exists a $\xi(n)$ such that 
\begin{equation}\label{C4}
\begin{aligned}
&f(x(n+1))-f(x(n))\\
&=\nabla_xf(x(n))(x(n+1)-x(n))+\frac{1}{2}(x(n+1)-x(n))H(\xi(n))(x(n+1)-x(n)).
\end{aligned}
\end{equation}
As $D$ is convex, we know that $\xi(n)$ also belongs to $D$ and thus we can apply \eqref{C3} to conclude that $|\partial_{x_i}\partial_{x_j}f(\xi(n))|\leq C$. Therefore, we combine \eqref{C1}, \eqref{C3} and \eqref{C4} to obtain
\begin{equation}\label{C5}
\begin{aligned}
f(x(n+1))-f(x(n))\leq-|\nabla_xf(x(n))|^2h+\frac{Ch^2}{2}|\nabla_xf(x(n))|^2.
\end{aligned}
\end{equation}
Then for sufficiently small $h$ such that $h<\frac{2}{C}$, \eqref{C5} implies that
\begin{equation}\label{C6}
\begin{aligned}
f(x(n+1))-f(x(n))\leq-|\nabla_xf(x(n))|^2h(1-\frac{Ch}{2})<0.
\end{aligned}
\end{equation}
The inequality \eqref{C6} shows that $f(x(n))$ is monotonic decreasing. On the other hand, as $x(n)$ is uniformly bounded in $D$ and $f(x)$ is analytic in $D$, we obtain that $f(x(n))$ is uniformly bounded. Therefore we can find a finite limit $f^{\infty}$ such that 
\begin{equation}\label{C7}
|\lim_{n\rightarrow +\infty}f(x(n))|=|f^{\infty}|<+\infty.
\end{equation}
Thus, combining \eqref{C7} and the continuity of $f$, we immediately obtain that 
\[f(x^{\infty})=\lim_{k\rightarrow+\infty}f(x(n_k))=f^{\infty}.\]
Moreover, we add up the inequality \eqref{C6} to obtain that 
\[f(x(0))-f(x(+\infty))=-\sum_{n=0}^{+\infty}(f(x(n+1))-f(x(n)))\geq\sum_{n=0}^{+\infty}|\nabla_xf(x(n))|^2h(1-\frac{Ch}{2}).\]
As both $f(x(0))$ and $f(x(+\infty)$ are finite and $h$ is sufficiently small, the finite of the sum of sequence $|\nabla_xf(x(n))|^2$ implies 
\[\lim_{n\rightarrow+\infty}|\nabla_xf(x(n))|=0.\]
Combining above formula, \eqref{C2} and the continuity of $|\nabla_xf(x(n))|$, we can conclude that 
\begin{equation}\label{C8}
|\nabla_xf(x^{\infty})|=\lim_{k\rightarrow+\infty}|\nabla_xf(x(n_k))|=0.
\end{equation}\newline

\noindent$\bullet$ (Step 3) In this step, we connect the discrete-time profile to the continuous time profile. According to Lemma \ref{L-inequality}, there exist $r>0$, $q>0$, and $\eta\in[\frac{1}{2},1)$ such that 
\begin{equation}\label{C9}
\|\nabla f(x)\|\geq q|f(x)-f(x^{\infty})|^{\eta},\quad \forall x\in B(x^{\infty},r).
\end{equation} 
On the other hand, as $f(x^{\infty})$ is finite, without loss of generality, we assume $f(x^{\infty})=0$. Then we let 
\begin{equation}\label{C10}
\bar{f}(t)=\frac{t-nh}{h}f(x(n+1))+\frac{(n+1)h-t}{h}f(x(n)),\quad nh\leq t\leq (n+1)h.
\end{equation} 
Due to the monotonic decreasing of $f(x(n))$, it is obviously that $\bar{f}(t)$ is Lipschitz continuous with respect to $t$ and monotonic decreasing to $f(x^{\infty})=0$. More precisely, we combine \eqref{C5} and \eqref{C10} and sufficiently small $h$ to have 
\begin{equation}\label{C11}
\frac{d}{dt}\bar{f}(t)=\frac{f(x(n+1))-f(x(n))}{h}\leq-|\nabla_xf(x(n))|^2(1-\frac{Ch}{2})<0,\quad \lim_{t\rightarrow+\infty}\bar{f}(t)=f(x^{\infty})=0.
\end{equation}
Next we set $g(t)=(\bar{f}(t))^{1-\eta}$ where $\eta$ is from \eqref{C9}. From \eqref{C10} and \eqref{C11}, we have $\lim\limits_{t\rightarrow+\infty}g(t)=0$ and  
\begin{equation}\label{C12}
\begin{aligned}
&\frac{d}{dt}g(t)=(1-\eta)(\bar{f}(t))^{-\eta}\frac{d}{dt}\bar{f}(t)\leq -(1-\eta)(\bar{f}(t))^{-\eta}|\nabla_xf(x(n))|^2(1-\frac{Ch}{2})<0.
\end{aligned}
\end{equation}
Then we apply \eqref{C9}, \eqref{C10}, \eqref{C12} and the decreasing of $f(x(n))$ to have an estimate of $|\nabla_xf(x(n))|$ in the interval $nh< t< (n+1)h$ as follows
\begin{equation}\label{C13} 
\begin{aligned}
&|\nabla_xf(x(n))|\\
&\leq -\frac{d}{dt}g(t)\frac{\left(\frac{t-nh}{h}f(x(n+1))+\frac{(n+1)h-t}{h}f(x(n)) \right)^{\eta}}{(1-\eta)|\nabla_xf(x(n))|(1-\frac{Ch}{2})}\\
&\leq -\frac{d}{dt}g(t)\frac{\left[f(x(n)) \right]^\eta}{(1-\eta)|\nabla_xf(x(n))|(1-\frac{Ch}{2})}\\
\end{aligned}
\end{equation}\newline

\noindent$\bullet$ (Step 4) In this step, we will prove that $\lim\limits_{n\rightarrow+\infty}x(n)=x^{\infty}$ by contradiction. Suppose not, then there exists a positive constant $l$ such that, for any $M$ there exists an integer $n_M\geq M$ satisfying 
\begin{equation}\label{C14}
|x(n_M)-x^{\infty}|\geq l.
\end{equation}  
Without loss of generality, we can assume $l$ is sufficiently small and $l\leq r$ where $r$ is in \eqref{C9}. Therefore, the $\L$ojasiewicz inequality in \eqref{C9} still holds in $B(x^{\infty},l)$. On the other hand, due to \eqref{C2}, \eqref{C7} and \eqref{C11}, we can find a sufficient large $n_0$ such that 
\begin{equation}\label{C15}
|x(n_0)-x^{\infty}|<\frac{l}{2},\quad |g(t)|<\frac{lq(1-\eta)(1-\frac{Ch}{2})}{4},\quad t\geq n_0h.
\end{equation}
Moreover, according to \eqref{C14}, we can find $n^*$ such that 
\begin{equation}\label{C16}
|x(n^*)-x^{\infty}|\geq l,\quad n^*>n_0.
\end{equation}
Now we consider the difference between $x(n^*)$ and $x(n_0)$. In fact, we apply \eqref{C1} to obtain  
\begin{equation}\label{C17}
\begin{aligned}
x(n^*)-x(n_0)=\sum_{i=n_0}^{n^*-1}(x(i+1)-x(i))=\sum_{i=n_0}^{n^*-1}(-\nabla_xf(x(i))h)=\sum_{i=n_0}^{n^*-1}\int_{ih}^{(i+1)h}(-\nabla_xf(x(i)))dt.
\end{aligned}
\end{equation}
Next, we combine \eqref{C9}, \eqref{C13}, \eqref{C15} and \eqref{C17} to obtain that
\begin{equation}\label{C18}
\begin{aligned}
&|x(n^*)-x(n_0)|\\
&\leq \sum_{i=n_0}^{n^*-1}\int_{ih}^{(i+1)h}\Big(-\frac{d}{dt}g(t)\frac{[f(x(i))]^\eta}{(1-\eta)|\nabla_xf(x(i))|(1-\frac{Ch}{2})}\Big)dt\\
&\leq \sum_{i=n_0}^{n^*-1}\int_{ih}^{(i+1)h}\Big(-\frac{d}{dt}g(t)\frac{1}{(1-\eta)q(1-\frac{Ch}{2})}\Big)dt\\
&= \sum_{i=n_0}^{n^*-1}\Big(\frac{g(ih)-g((i+1)h)}{(1-\eta)q(1-\frac{Ch}{2})}\Big)\\
&=\frac{g(n_0h)-g(n^*h)}{(1-\eta)q(1-\frac{Ch}{2})}\leq \frac{l}{2}.
\end{aligned}
\end{equation}
Combining \eqref{C15} and \eqref{C18}, we obtain that 
\begin{equation}\label{C19}
|x(n^*)-x^{\infty}|\leq |x(n^*)-x(n_0)|+|x(n_0)-x^{\infty}|<l,
\end{equation}
which is obvious contradicted to \eqref{C16}. Therefore, we conclude that  $\lim\limits_{n\rightarrow+\infty}x(n)=x^{\infty}$.\newline
\end{proof}
\begin{remark}\label{R3.1}
The identical and non-identical Kuramoto model can be treated as a gradient flow on the circle and real line, respectively. As circle is a compact manifold, we can directly apply Theorem \ref{T3.1} to conclude the existence of the asymptotical equilibrium and the corresponding stability for identical Kuramoto model. While for non-identical Kuramoto model, we can obtain the stability once we show the uniform boundedness of the oscillators on the real line. However, in order to obtain the convergence rate, we need to analyze the $\L$ojasiewicz exponent which is far from trivial.\newline 
\end{remark}
%then, add \eqref{pi is the limit of subtraction between oscillators in I_s and I_b} from this, we obtain
%\begin{equation}\label{2.5}
%\lim_{t \to +\infty} \theta_j(t) = -\frac{\pi}{N}, \qquad j \in \mathcal{I}_s.
%\end{equation}
%By adding \eqref{pi is the limit of subtraction between oscillators in I_s and I_b} from this equation, we obtain
%\begin{equation}\label{|I_b| is 1, then (N-1)pi/N is the limit of oscillators in I_b}
%\lim_{t \to +\infty} \theta_N(t) = \frac{(N-1)\pi}{N}, \qquad N \in \mathcal{I}_b. 
%\end{equation}
%For the situation that 
%\[\lim_{t \to \infty} (\theta_N(t) - \theta_j(t)) = -\pi, \quad \text{for all } j \in \mathcal{I}_s,\]
%using the same argument as above, we have
%\[\lim_{t \to +\infty} \theta_j(t) = \frac{\pi}{N}, \qquad j \in \mathcal{I}_s,44\]
%and
%\[\lim_{t \to +\infty} \theta_N(t) = -\frac{(N-1)\pi}{N}, \qquad N \in \mathcal{I}_b.\]

%Next, we review the previous result on the continuous non-identical Kuramoto model. The following proposition states that any phase configuration whose initial phase diameter is in the range $(0, \pi)$  will shrink to a smaller set $\{\Theta \in \mathbb{T}^{N} : D(\Theta) \le D^\infty\}$
\section{Discrete identical Kuramoto model}\label{sec:4}
\setcounter{equation}{0}
\vspace{0.5cm}
In this section, we will pay attention to the identical discrete Kuramoto model and its large time behavior. Due to the conservation of the mean natural frequency, without loss of generality, we may assume the natural frequencies for all particles are identically equal to zero. According to Theorem \ref{T3.1} and Remark \ref{R3.1}, although we can conclude the convergence to the equilibrium state for identical oscillators, we cannot figure out any convergence rate. Therefore, in this section, we will choose another manner to attain the convergence and the corresponding rate.\newline 

For the case $N=2$, the two particles are always contained in a half circle. Then, it is obviously that phase synchronization will emerge except for initial data such that $|\theta_1(0)-\theta_2(0)|=\pi$. Therefore, we will only discuss the case when $N\geq3$. According to Remark \ref{R2.1}, almost all initial data $\Theta_0$ satisfy the condition \eqref{B2}. Thus, it is reasonable for us to study the discrete model with the initial configuration with property \eqref{B2}. Then, according to Lemma \ref{L2.2}, the bipolar set $\mathcal{I}_b$ has cardinality no more than one in the continuous Kuramoto model. Therefore, we can split the set of initial data with property \eqref{B2} into two subsets,  
\begin{equation}\label{D1}
\mathcal{A}_1 := \left\{\Theta_0 \left|\right. |\mathcal{I}_b(\Theta_0)| = 0,\ \mbox{$\Theta_0$ satisfies \eqref{B2}}\right\}, \quad  \mathcal{A}_2 := \left\{\Theta_0 \left|\right. |\mathcal{I}_b(\Theta_0)| = 1,\ \mbox{$\Theta_0$ satisfies \eqref{B2}}\right\}.
\end{equation}
Then, applying the conservation of the total phases, we can further construct the asymptotical limits in Remark \ref{R2.1}, based on the settings $\mathcal{I}_s$ and $\mathcal{I}_b$ in Lemma \ref{L2.2}. More precisely, for case $\mathcal{A}_1$, we can represent the limits of $\phi(t)$ and $\theta_j(t)$ as below
\begin{equation}\label{D2}
\lim_{t \to + \infty} \phi(t) = -\frac{1}{N} \sum_{j=1}^N 2k_j \pi := \phi_0^*, \quad \lim_{t \to +\infty} \theta_j(t) = 2k_j \pi + \phi_0^*, \quad k_j \in \bbz, \ j =1, \ldots, N.
\end{equation}
Similarly, for case $\mathcal{A}_2$, without loss of generality, we assume that $\mathcal{I}_b = \{N\}$. Then the limits of $\phi(t)$ and $\theta_j(t)$ can be further represented as
\begin{align}\label{D3}
\left\{\begin{aligned}
&\lim_{t \to + \infty} \phi(t) =  -\frac{1}{N} [\sum_{j=1}^{N-1} 2k_j \pi + (2k_N + 1) \pi] := \phi_1^*\\
&\lim_{t \to + \infty} \theta_j(t) = 2k_j \pi + \phi_1^*, \ j =1,\ldots, N-1, \qquad \lim_{t \to + \infty} \theta_N(t) = (2k_N+1) \pi + \phi_1^*.
\end{aligned}
\right.
\end{align}
Note that all the above analysis are based on the continuous Kuramoto model, but the classification \eqref{D1} of initial data can still be applied to the study of discrete Kuramoto model. Then we have the following theorem for identical discrete Kuramoto model.
\begin{theorem}\label{T4.1}
(Identical Kuramoto) For $N\geq 3$, we let $\Theta^h(n) = (\theta^h_1(n), \ldots, \theta^h_N(n))$ be a solution to the discrete identical Kuramoto model \eqref{A-2} with initial phase $\Theta_0$ satisfying the following conditions:
\[\Omega_i=0,\quad \frac{1}{N} \sum_{j=1}^N\theta_j^0=0,  \quad r_0 > 0, \quad \theta_i^0 \neq \theta_j^0, \quad i \ne j, \quad 1 \le i ,j \le N.\]
Then, with the time-step $h$ sufficiently small, there exists a phase-locked state $\Theta^\infty$ , constants $C>0, \alpha > 0$, and a step $n_e > 0$ such that
\[\| \Theta^h(n) - \Theta^\infty\|_\infty < Ce^{-\alpha (n-n_e)h}, \qquad n\ge n_e,\]
where $C,\alpha$ are both dependent on the initial configuration. Moreover, all the possible asymptotic phase-locked states can be expressed as follows:
\begin{enumerate}
\item $\Theta^\infty =(2k_1 \pi + \phi_0^*,\ldots, 2k_{N}\pi + \phi_0^*)$, or 
\item $\Theta^\infty = (2k_1\pi + \phi_1^*, \ldots, 2k_{N-1} \pi + \phi_1^*, (2k_N+1)\pi + \phi_1^*)$,
\end{enumerate}
where $k_i \in \bbz, \ i =1,2,\ldots,N.$
\end{theorem}
In the following, we will prove Theorem \ref{T4.1} by studying the dynamics of discrete Kuramoto model with initial data in $\mathcal{A}_1$ and $\mathcal{A}_2$ respectively.\newline 

\subsection{Case $\mathcal{A}_1$ (phase synchronization)} \label{sec:4.1}
For initial data $\Theta_0\in\mathcal{A}_1$, we have $|I_b| = 0$. Then it follows from \eqref{D2} that for any given $\varepsilon > 0$ and any $i \in \mathcal{I}_s$, we can find a time $T_\varepsilon > 0$ and a equilibrium state $\phi_0^*$ such that 
 \begin{equation}\label{c-each limit}
|\theta_i(T_\varepsilon) - 2k_i \pi - \phi_0^*| < \varepsilon, \qquad i \in \mathcal{I}_s,
\end{equation}
where $\mathcal{I}_s$ is defined in Lemma \ref{L2.2} and $\theta_i(t)$ is the solution to the continuous Kuramoto model \eqref{A-1} with initial data $\Theta_0$ and identical natural frequencies. Then, the solution to the descrete identical Kuromoto model \eqref{A-2} on the interval $[0, T_\varepsilon]$ can be well approximated by the solution to the continuous model with same initial data, provided the step size $h$ is sufficiently small. Therefore, we can obtain a good estimate on the oscillators for discrete-time model based on the approximation.\newline 

Then we define the effective phase $\hat{\Theta}^h(n)$ for the discrete identical Kuramoto model with respect to the initial data $\Theta_0$ as below,
\begin{equation}\label{D5}
\hat{\theta}^h_i(n) = \theta_i^h(n) - 2k_i\pi - \phi_0^*, \qquad n = 0,1,\ldots \quad i=1,2,\ldots,N.
\end{equation}
Note in this case, every oscillator belongs to $\mathcal{I}_s$. Then, the effective phases for the oscillators in the synchronization group $\mathcal{I}_s$, can be defined as follows,  
\begin{align}\label{D6}
\left\{\begin{aligned}
&\hat{\Theta}^h_s := (\hat{\theta}^h_{i_1},\ldots, \hat{\theta}^h_{i_{|\mathcal{I}_s|}}), \quad i_k \in \mathcal{I}_s,\quad \hat{\theta}^h_M := \max_{j \in \mathcal{I}_s} \hat{\theta}^h_j, \quad \hat{\theta}^h_m := \min_{j \in \mathcal{I}_s} \hat{\theta}^h_j, \\
&D(\hat{\Theta}^h_s) := \max_{i,j \in \mathcal{I}_s} |\hat{\theta}^h_i - \hat{\theta}^h_j|= \hat{\theta}^h_M - \hat{\theta}^h_m.
\end{aligned}
\right.
\end{align}

\begin{remark}\label{R4.1}
The definition of the effective phase depends on $\phi_0^*$ which is the asymptotical limits of the oscillators with initial data in $\mathcal{A}_1$. Therefore, the effective phases actually depend on the initial data. More precisely, for any given fixed initial data satisfying \eqref{B2}, we can define corresponding effective phases. 
\end{remark}

\begin{lemma}\label{L4.1}
For $N\geq 3$, we let $\Theta^h(n)=(\theta^h_1(n), \theta^h_2(n), \ldots, \theta^h_N(n))$ be a solution to the discrete identical Kuramoto model \eqref{A-2} with initial data $\Theta_0\in\mathcal{A}_1$. Then for any given $\varepsilon > 0$ and sufficiently small step size $h\ll 1$, we can find a positive integer $l$ such that
\[D(\hat{\Theta}^h_s(l)) < \varepsilon,\] 
where $\mathcal{A}_1$ and $\hat{\Theta}^h_s$ are defined in \eqref{D1} and \eqref{D6} respectively.
\end{lemma}
\begin{proof}
We will show the proof by two steps.

\noindent $\bullet$ (Step 1): In this step, to avoid confusion, we will denote the solution to discrete-time model by $\theta_i^h(n)$ and the solution to the continuous time model by $\theta_i(t)$. According to \eqref{D1}, \eqref{D2}, \eqref{D3} and \eqref{c-each limit}, for a given initial data $\Theta_0\in\mathcal{A}_1$ and positive constant $\varepsilon$, we can find a time $T_\varepsilon > 0$ and a equilibrium state $\phi_0^*$ such that \eqref{c-each limit} holds i.e.
 \begin{equation*}
|\theta_i(T_\varepsilon) - 2k_i \pi - \phi_0^*| < \frac{\varepsilon}{4}, \qquad i \in \mathcal{I}_s,
\end{equation*}
where $\theta_i(T_\varepsilon)$ is the solution to the continuous Kuramoto model \eqref{A-1} at time $T_\varepsilon$. On the other hand, based on the result in Lemma \ref{L2.5} about the global error in Euler's scheme, there exists constants $M > 0$ and $L_M > 0$ which are both independent of $i = 1,2,\ldots,N$ such that
\[|\theta_i(nh) - \theta_i^h(n)| \le Me^{L_M T_\varepsilon}h, \qquad n = 0,1,\ldots,l, \quad i = 1,2,\ldots,N,\]
where $h = \frac{T_\varepsilon}{l}$ and $l$ is a positive integer. Then for the above given $\varepsilon$ in \eqref{c-each limit}, we choose the step size $h$ sufficiently small such that $Me^{L_M T_\varepsilon}h < \frac{\varepsilon}{4}$. Moreover, we require $\frac{T_\varepsilon}{h}$ is an integer and denote it by $l$. Therefore, we have
\begin{equation}\label{D-7}
|\theta_i(nh) - \theta_i^h(n)| < \frac{\varepsilon}{4}, \qquad n = 0,1,\ldots,l, \quad i =1,2,\ldots,N.
\end{equation}
When $n = l$, we have $T_\varepsilon=lh$ and thus we apply \eqref{c-each limit} and \eqref{D-7} to obtain that
\begin{equation}\label{d-each limit}
|\theta_i^h(l) - 2k_i\pi - \phi_0^*| \le |\theta_i(T_\varepsilon) - \theta_i^h(l)| + |\theta_i(T_\varepsilon) - 2k_i\pi - \phi_0^*| < \frac{\varepsilon}{2}, \quad i \in \mathcal{I}_s.
 \end{equation}

\noindent $\bullet$ (Step 2): In this step, we will use $\theta_i^h(n)$ to denote the $n$-th step of the solution to the discrete-time model. According to the definition of effective phase in \eqref{D5} and the estimate \eqref{d-each limit}, we immediately have
\[|\hat{\theta}^h_i(l)| < \frac{\varepsilon}{2}, \quad i \in \mathcal{I}_s.\]
Hence, we find a step $l$ such that
\[D(\hat{\Theta}^h_s(l)) = \max_{i,j \in \mathcal{I}_s} |\hat{\theta}^h_i(l) - \hat{\theta}^h_j(l)| < \varepsilon, \quad i,j \in \mathcal{I}_s.\]
\end{proof}
According to Lemma \ref{L4.1},  in order to study the large time behavior of the discrete-time Kuramoto model, we can set  $l$ to be the initial step. Therefore, without loss of generality, we may assume the initial configuration satisfying the properties below,
\begin{equation}\label{D9}
 D(\hat{\Theta}^h_s(0)) = \hat{\theta}^h_N(0) - \hat{\theta}^h_1(0) < \varepsilon,\quad \hat{\theta}^h_N(0) > \hat{\theta}^h_{N-1}(0) > \cdots > \hat{\theta}^h_1(0).
 \end{equation}
Moreover, according to the discrete Kuramoto model \eqref{A-2}, initial configuration $\Theta_0$ satisfying \eqref{B2} and the definition of effective phase $\hat{\Theta}^h(n)$ in \eqref{D5}, we immediately conclude that $\hat{\Theta}^h(n)$ satisfies the identical discrete Kuramoto model below
\begin{equation}\label{hat system}
\left\{\begin{aligned}
&\hat{\theta}^h_i(n+1) = \hat{\theta}^h_i(n) + \frac{Kh}{N} \sum_{j=1}^N \sin(\hat{\theta}^h_j(n) - \hat{\theta}^h_i(n)), \quad i = 1,2,\ldots,N,\\
&\sum_{i=1}^N\hat{\theta}^h_i(0)=0.
\end{aligned}
\right.
\end{equation}

\begin{lemma}\label{L4.2} 
For $N\geq 3$, we let $\hat{\Theta}^h(n) = (\hat{\theta}^h_1(n), \ldots, \hat{\theta}^h_N(n))$ be a solution to the discrete identical Kuramoto model \eqref{hat system} with the initial configuration satisfying \eqref{D9} i.e.
\begin{equation}\label{D11}
 D(\hat{\Theta}^h_s(0)) = \hat{\theta}^h_N(0) - \hat{\theta}^h_1(0) < \varepsilon,\quad \hat{\theta}^h_N(0) > \hat{\theta}^h_{N-1}(0) > \cdots > \hat{\theta}^h_1(0).
 \end{equation}
where $\varepsilon$ is a sufficient small positive constant. Then we conclude that the effective phase diameter is uniform bounded by the same $\varepsilon$ in \eqref{D11} and the order of the effective phase will be preserved for all $n$ i.e.
\begin{align*}
\left\{\begin{aligned}
&D(\hat{\Theta}^h_s(n)) < \varepsilon, \qquad n = 0, 1,2,\ldots,\\
&\hat{\theta}^h_N(n) > \hat{\theta}^h_{N-1}(n) > \cdots > \hat{\theta}^h_1(n), \qquad n = 0, 1,2,\ldots.
\end{aligned}
\right.
\end{align*}
\end{lemma}
\begin{proof}
\noindent $\bullet$ (Step $1$): In the first step, we assume $D(\hat{\Theta}^h_s(n)) < \varepsilon$  holds for all step $n$, then we claim that the order of the oscillators is preserved for all steps n i.e.
\begin{equation}\label{D12} 
\hat{\theta}^h_N(n) > \hat{\theta}^h_{N-1}(n) > \cdots > \hat{\theta}^h_1(n), \qquad n = 0,1,2,\ldots.
\end{equation}
Actually, according to \eqref{D11}, the order \eqref{D12} automatically holds for the  initial step $\hat{\Theta}(0)$. Now we assume that the order \eqref{D12} is preserved for step $k$, i.e.
\begin{equation}\label{D13}
\hat{\theta}^h_{N}(k) > \hat{\theta}^h_{N-1}(k) > \cdots > \hat{\theta}^h_{1}(k).
\end{equation}
Then, for the step $k+1$, the difference of $\hat{\theta}^h_i(k+1)$ and $\hat{\theta}^h_{i-1}(k+1)$, where $i =2,3, \cdots, N,$ satisfies the following equation, 
\begin{align}\label{D14}
\begin{aligned}
&\hat{\theta}^h_{i}(k+1) -\hat{\theta}^h_{i-1}(k+1) \\
&= \hat{\theta}^h_{i}(k) -\hat{\theta}^h_{i-1}(k) + \frac{Kh}{N} \sum^{N}_{j=1} \sin(\hat{\theta}^h_j(k) - \hat{\theta}^h_i(k)) - \frac{Kh}{N} \sum^{N}_{j=1}  \sin(\hat{\theta}^h_j(k) - \hat{\theta}^h_{i-1}(k))\\
&=\hat{\theta}^h_{i}(k) -\hat{\theta}^h_{i-1}(k) -\frac{Kh}{N}\sum^{i-1}_{j=1} \left[ \sin(\hat{\theta}^h_i(k) - \hat{\theta}^h_j(k)) - \sin(\hat{\theta}^h_{i-1}(k) - \hat{\theta}^h_j(k)) \right] \\
&-\frac{Kh}{N}\sum^{N}_{j=i} \left[ \sin(\hat{\theta}^h_j(k) - \hat{\theta}^h_{i-1}(k)) - \sin(\hat{\theta}^h_j(k) - \hat{\theta}^h_i(k)) \right]\\
&=\mathcal{I}_1+\mathcal{I}_2+\mathcal{I}_3.
\end{aligned}
\end{align}
%		For the last two terms, we have
%		\begin{align*}
%		\begin{aligned}
%		& \sum^{N}_{j=1} [\sin(\hat{\theta}^h_j(n) -\hat{\theta}^h_i(n)) - \sin(\hat{\theta}^h_j(n) - \hat{\theta}_{i-1}(n))]  \\
%		&= \sum^{i-1}_{j=1} \left[ \sin(\hat{\theta}^h_j(n) - \hat{\theta}^h_i(n)) - \sin(\hat{\theta}^h_j(n) - \hat{\theta}_{i-1}(n)) \right] \\
%		&\hspace{0.5cm} + \sum^{N}_{j=i} \left[ \sin(\hat{\theta}^h_j(n) - \hat{\theta}^h_i(n)) - \sin(\hat{\theta}^h_j(n) - \hat{\theta}_{i-1}(n)) \right] \\
%		&= - \sum^{i-1}_{j=1} \left[ \sin(\hat{\theta}^h_i(n) - \hat{\theta}^h_j(n)) - \sin(\hat{\theta}_{i-1}(n) - \hat{\theta}^h_j(n)) \right] \\
%		&\hspace{0.5cm} - \sum^{N}_{j=i} \left[ \sin(\hat{\theta}^h_j(n) - \hat{\theta}_{i-1}(n)) - \sin(\hat{\theta}^h_j(n) - \hat{\theta}^h_i(n)) \right] .
%		\end{aligned}
%		\end{align*}
The first term $\mathcal{I}_1$ is obviously positive due to the assumption \eqref{D13}. For $\mathcal{I}_2$ and $\mathcal{I}_3$, we note that $i$ and $i-1$ are two adjacent oscillators. Therefore, according to the order \eqref{D13}, the differences $\hat{\theta}^h_i(k) - \hat{\theta}^h_j(k)$ and $\hat{\theta}^h_{i-1}(k) - \hat{\theta}^h_j(k)$ are of the same sign for all $j\neq i,\ i-1$. Moreover, as $D(\hat{\Theta}^h_s(k)) < \varepsilon\ll 1$, the difference  $\hat{\theta}^h_i(k) - \hat{\theta}^h_j(k)$ and $\sin(\hat{\theta}^h_i(k) - \hat{\theta}^h_j(k))$ are of the same sign. Then, we can apply Lemma \ref{sub-additive} to obtain 
\begin{align}\label{D15}
\left\{\begin{aligned}
& - \sum^{i-1}_{j=1} \left[ \sin(\hat{\theta}^h_i(k) - \hat{\theta}^h_j(k)) - \sin(\hat{\theta}^h_{i-1}(k) - \hat{\theta}^h_j(k)) \right]  \geq -(i-1) \sin(\hat{\theta}^h_i(k) - \hat{\theta}^h_{i-1}(k))),\\
& - \sum^{N}_{j=i} \left[ \sin(\hat{\theta}^h_j(k) - \hat{\theta}^h_{i-1}(k)) - \sin(\hat{\theta}^h_j(k) - \hat{\theta}^h_i(k)) \right] \geq - (N - i +1) \sin(\hat{\theta}^h_i(k) - \hat{\theta}^h_{i-1}(k)).
\end{aligned}
\right.
\end{align}
We combine \eqref{D14} and \eqref{D15} and apply the inequality $\sin x < x$ for $x > 0$ to obtain the estimate of the oscillator difference at $k+1$ step as below,
%		\begin{align*}
%		\begin{aligned}
%		& \frac{Kh}{N} \sum^{N}_{j=1} \left[ \sin(\hat{\theta}^h_j(n) - \hat{\theta}^h_i(n)) - \sin(\hat{\theta}^h_j(n) - \hat{\theta}_{i-1}(n)) \right] \\
%		&\geq  -Kh \sin(\hat{\theta}^h_i(n) - \hat{\theta}_{i-1}(n)).
%		\end{aligned}
%		\end{align*} 
\begin{align*}
\begin{aligned}
&\hat{\theta}^h_{i}(k+1) - \hat{\theta}^h_{i-1}(k+1)\\
&> \hat{\theta}^h_{i}(k) -\hat{\theta}^h_{i-1}(k) - Kh(\hat{\theta}^h_i(k) - \hat{\theta}^h_{i-1}(k))= (1 - Kh)(\hat{\theta}^h_{i}(k) -\hat{\theta}^h_{i-1}(k))
\end{aligned}
\end{align*}
Now, we can choose $h$ sufficiently small so that $1 - Kh >0$. Then, according to \eqref{D13}, we have 
\[ \hat{\theta}^h_{i}(k+1) - \hat{\theta}^h_{i-1}(k+1) >0, \qquad i = 2,3, \ldots, N.\]
Therefore, it follows by induction principle that the order \eqref{D12} holds for each step $n$ and we finish the proof of our claim.\newline

\noindent $\bullet$ (Step $2$): In this step, we claim that $D(\hat{\Theta}^h_s(n)) < \varepsilon$  holds for all step $n$. We will prove our claim by contradiction. Actually, suppose the inequality $D(\hat{\Theta}^h_s(n)) < \varepsilon$ does not hold for all $n$. Then, there exists the ``stoping step" $n_0$ such that 
\begin{equation}\label{D16}
\left\{
\begin{aligned}
&D(\hat{\Theta}^h_s(n)) < \varepsilon\ll 1, \quad 0 \le n \le n_0,\\
&D(\hat{\Theta}^h_s(n_0+1))  \ge \varepsilon.
\end{aligned}
\right.
\end{equation} 
Using the same argument as in step $1$, we can immediately conclude that the order \eqref{D12} is preserved for each step $0 \leq n \leq n_0 + 1$, i.e.,
\begin{equation}\label{D17}
\hat{\theta}^h_{i}(n) > \hat{\theta}^h_{i-1}(n), \qquad i = 2,3, \cdots, N, \quad 0 \leq n \leq n_0 + 1.
\end{equation}
Then, according to \eqref{D17}, the diameter can be represented as $D(\hat{\Theta}^h_s(n))=\hat{\theta}^h_{N}(n) - \hat{\theta}^h_{1}(n)$ for all steps $0 \leq n \leq n_0 + 1$. Therefore, the phase diameter of the $(n_0 +1)$-th step can be obtained as below,
\begin{align}\label{D18}
\begin{aligned}
 &D(\hat{\Theta}^h_s(n_0 + 1))= \hat{\theta}^h_{N}(n_0 +1) - \hat{\theta}^h_{1}(n_0 +1) \\
&\hspace{0.5cm}= \hat{\theta}^h_{N}(n_0) - \hat{\theta}^h_{1}(n_0) + \frac{Kh}{N} \sum^{N}_{j=1} \sin(\hat{\theta}^h_{j}(n_0) - \hat{\theta}^h_{N}(n_0))  - \frac{Kh}{N} \sum^{N}_{j=1} \sin(\hat{\theta}^h_{j}(n_0) - \hat{\theta}^h_{1}(n_0)).
\end{aligned}
\end{align}
%		Applying Lemma \ref{sub-additive} to estimate the last two terms, we have
%		\begin{align*}
%		\begin{aligned}
%		& \sum^{N}_{j=1} \left[ \sin(\hat{\theta}^h_{j}(n_0) - \hat{\theta}^h_{N}(n_0)) - \sin(\hat{\theta}^h_{j}(n_0) - \hat{\theta}^h_{1}(n_0)) \right] \\
%		&= - \sum^{N}_{j=1} \left[ \sin(\hat{\theta}^h_{N}(n_0) - \hat{\theta}^h_{j}(n_0)) + \sin(\hat{\theta}^h_{j}(n_0)- \hat{\theta}^h_{1}(n_0)) \right] \\
%		&\le -N\sin(\hat{\theta}^h_{N}(n_0) - \hat{\theta}^h_{1}(n_0)),
%		\end{aligned}
%		\end{align*}
Then, due to \eqref{D16} and \eqref{D17}, we can apply the same argument in step $1$ and Lemma \ref{sub-additive} to obtain that 
\begin{equation}\label{D19}
\sum^{N}_{j=1} \sin(\hat{\theta}^h_{j}(n_0) - \hat{\theta}^h_{N}(n_0)) - \sum^{N}_{j=1} \sin(\hat{\theta}^h_{j}(n_0) - \hat{\theta}^h_{1}(n_0))\leq -N\sin(\hat{\theta}^h_{N}(n_0) - \hat{\theta}^h_{1}(n_0))<0.
\end{equation}
Therefore, we combine \eqref{D18} and \eqref{D19} to obtain that 
\[\hat{\theta}^h_{N}(n_0 +1) - \hat{\theta}^h_{1}(n_0 +1) < \hat{\theta}^h_{N}(n_0) - \hat{\theta}^h_{1}(n_0) < \varepsilon,\]
which is contradicted to $\eqref{D16}_2$. Therefore, we finish the proof of our claim and conclude that $D(\hat{\Theta}^h_s(n)) < \varepsilon$  holds for all step $n$.
\end{proof}
Combining Lemma \ref{L4.1} and Lemma \ref{L4.2}, we know that, for all initial data $\Theta_0\in \mathcal{A}_1$, the diameter of effective phase and the corresponding order will be invariant after a particular time $l$. Next, we will study the asymptotic synchronization behaviors of oscillators in $\mathcal{I}_s = \{1,2,\ldots,N\}$ for discrete system. The following result states that the convergence to zero of effective phase diameter is at least exponential, which implies that the effective phases of all oscillators in $\mathcal{I}_s$ will converge to zero at least exponentially.

\begin{lemma}\label{L4.3}
For $N\geq 3$, we let $\hat{\Theta}^h = (\hat{\theta}^h_1(n), \ldots, \hat{\theta}^h_N(n))$ be a solution to the discrete identical Kuramoto model \eqref{hat system} with the initial configuration satisfying $D(\hat{\Theta}^h_s(0)) < \varepsilon$, where $\varepsilon$ is a sufficiently small positive real number. Then, there exists a positive number $h_0 > 0$ such that if $0<h<h_0$, we have 
\begin{equation}\label{D20}
\left\{
\begin{aligned}
&D(\hat{\Theta}^h_s(n)) < D(\hat{\Theta}^h_s(0)) \exp \left(-\frac{K\sin \varepsilon}{2\varepsilon} nh\right), \\
& |\hat{\theta}^h_j(n)| < D(\hat{\Theta}^h_s(0)) \exp \left(-\frac{K\sin \varepsilon}{2\varepsilon} nh\right), \qquad n = 0,1,2,\ldots, \qquad j \in \mathcal{I}_s.
\end{aligned}
\right.
\end{equation}
\end{lemma}
\begin{proof}
We can apply Lemma \ref{sub-additive} and the same argument as in the proof of Lemma \ref{L4.2} to obtain that
\begin{equation}\label{D21}
D(\hat{\Theta}^h_s(n +1)) = \hat{\theta}^h_{N}(n+1) - \hat{\theta}^h_{1}(n +1) \le \hat{\theta}^h_{N}(n) - \hat{\theta}^h_{1}(n) - Kh\sin(\hat{\theta}^h_{N}(n) - \hat{\theta}^h_{1}(n)).
\end{equation}
Since the function $\frac{\sin x}{x} $ is monotonically decreasing in $[0, \varepsilon]$ when $\varepsilon$ is sufficiently small, we can apply Lemma \ref{L4.2} to obtain that 
\begin{equation}\label{D22}
 D(\hat{\Theta}^h_s(n))= \hat{\theta}^h_{N}(n) - \hat{\theta}^h_{1}(n)< \varepsilon,\quad \frac{\sin(\hat{\theta}^h_{N}(n) - \hat{\theta}^h_{1}(n))}{\hat{\theta}^h_{N}(n) - \hat{\theta}^h_{1}(n)}  > \frac{\sin \varepsilon}{\varepsilon}.
 \end{equation}
Hence, we combine \eqref{D21} and \eqref{D22} to obtain 
\begin{equation}\label{D23}
\hat{\theta}^h_{N}(n+1) - \hat{\theta}^h_{1}(n +1) < (1 - Kh \frac{\sin \varepsilon}{\varepsilon})(\hat{\theta}^h_{N}(n) - \hat{\theta}^h_{1}(n)).
\end{equation}
Then, the iteration of \eqref{D23} leads to the estimate of the diameter of effective phases at step $n$ as below,
\begin{equation}\label{D24}
D(\hat{\Theta}^h_s(n))  < \left(1 - Kh \frac{\sin \varepsilon}{\varepsilon}\right)^n (\hat{\theta}^h_{N}(0) - \hat{\theta}^h_{1}(0)) = D(\hat{\Theta}^h_s(0)) \exp \left[ nh \frac{\log(1 - Kh \frac{\sin \varepsilon}{\varepsilon})}{h}\right].
\end{equation}
Now, we can choose the step size sufficiently small to guarantee $\left(1 - Kh \frac{\sin \varepsilon}{\varepsilon}\right)>0$, so that the last term in \eqref{D24} is well defined. Moreover, according to L'Hospital's rule, we have 
\begin{equation}\label{D25}
\lim_{h \to 0} \frac{\log(1 - Kh \frac{\sin \varepsilon}{\varepsilon})}{h} = -K \frac{\sin \varepsilon}{\varepsilon}.
\end{equation}
The estimates \eqref{D24} and \eqref{D25} implies that there exists a positive constant $h_0$ such that if $0<h<h_0$, we have
\begin{equation}\label{D26}
\frac{\log(1 - Kh \frac{\sin \varepsilon}{\varepsilon})}{h} < -\frac{K\sin \varepsilon}{2\varepsilon},\quad D(\hat{\Theta}^h_s(n)) < D(\hat{\Theta}^h_s(0)) \exp \left(-\frac{K\sin \varepsilon}{2\varepsilon} nh\right), \quad n\geq 0.
\end{equation}
Moreover, from \eqref{hat system} and the fact $\sum\limits_{j=1}^{N} \hat{\theta}^h_j(0) = 0$, we immediately obtain that  $\sum\limits_{j=1}^{N} \hat{\theta}^h_j(n) = 0$ holds for all steps $n$. Therefore, \eqref{D26} implies that 
\[	|\hat{\theta}^h_j(n)| = \left| \hat{\theta}^h_j(n) - \frac{\sum_{i =1}^N \hat{\theta}^h_i(n)}{N} \right| \le \frac{\sum_{i =1}^N |\hat{\theta}^h_j(n) - \hat{\theta}^h_i(n)|}{N} \le D(\hat{\Theta}^h_s(n)), \]
which finish the proof of the lemma.
\end{proof}

\begin{remark}\label{R4.2}
For initial data $\Theta_0\in\mathcal{A}_1$, we can construct $\phi_0^*$ as in \eqref{D2} based on the asymptotical behavior of continuous time model. Then the relationship between $\hat{\theta}^h_i(n)$ and $\theta_i^h(n)$ can be written as below
\[\hat{\theta}^h_i(n) = \theta_i^h(n) - 2k_i \pi - \phi_0^*, \quad i \in \mathcal{I}_s.\]
Note $k_i$ and $\phi_0^*$ satisfying \eqref{D2} i.e. $\lim\limits_{t\rightarrow+\infty}|\theta_i(t)-2k_i \pi - \phi_0^*|=0$, where $\theta_i(t)$ is the solution to the continuous model with initial data $\Theta_0\in\mathcal{A}_1$. Therefore, Lemma \ref{L4.1}, Lemma \ref{L4.2}, Lemma \ref{L4.3} and \eqref{D2} together show that, for initial data $\Theta_0\in\mathcal{A}_1$, the solution to the discrete identical Kuramoto model and the solution to the continuous Kuramoto model have the same asymptotical limits, i.e.
\[\lim_{n\rightarrow+\infty} |\theta_i^h(n) - 2k_i \pi - \phi_0^*|=\lim\limits_{t\rightarrow+\infty}|\theta_i(t)-2k_i \pi - \phi_0^*|=0, \quad i \in \mathcal{I}_s.\]\newline
\end{remark}

\subsection{Case $\mathcal{A}_2$ (bipolar formation)} \label{sec:4.2}
For initial data $\Theta_0\in\mathcal{A}_2$, we have $|\mathcal{I}_b| = 1$ and $|\mathcal{I}_s| = N - 1$. Without loss of generality, we suppose that
\[ \mathcal{I}_s = \{1,2,\ldots, N-1\}, \quad \mathcal{I}_b = \{N\}.\]
From \eqref{D3}, for any given positive constant $\varepsilon\ll 1$, we can find a large enough time $T_\varepsilon > 0$ such that all oscillators are closed to their corresponding limit, i.e.
\begin{align}\label{D27}
\left\{\begin{aligned}
&\lim_{t \to + \infty} \phi(t) =  -\frac{1}{N} [\sum_{j=1}^{N-1} 2k_j \pi + (2k_N + 1) \pi] := \phi_1^*\\
& |\theta_N(T_\varepsilon) - (2k_N+1) \pi - \phi_1^*|<\frac{\varepsilon}{4},\quad N\in \mathcal{I}_b,\\
&|\theta_j(T_\varepsilon) - 2k_j \pi - \phi_1^*|<\frac{\varepsilon}{4},\quad  j \in\mathcal{I}_s.
\end{aligned}
\right.
\end{align}	
Similar as \eqref{D5}, for the discrete-time Kuramoto model, we can define the effective phase for oscillators in $\mathcal{I}_s$	 and $\mathcal{I}_b$ respectively. More precisely, we let 
\begin{equation}\label{D28}
\hat{\theta}^h_i(n) = \theta_i^h(n) - 2k_i\pi - \phi_1^* - \frac{\pi}{N}, \quad n=0,1,2,\cdots\quad i = 1,2,\ldots,N.
\end{equation}
According to \eqref{D27} and \eqref{D28}, it is obvious that $\hat{\theta}^h_i(n)$ satisfy the identical Kuramoto model \eqref{hat system}. Then, similar as Lemma \ref{L4.1}, we can prove that the solution of the discrete Kuramoto model with initial data $\Theta_0\in\mathcal{A}_2$ is closed to the solution of the continuous model with same initial data. More precisely, we define the effective phases for the oscillators in the synchronization group $\mathcal{I}_s$ as follows,  
\begin{align}\label{D29}
\left\{\begin{aligned}
&\hat{\Theta}^h_s := (\hat{\theta}^h_{i_1},\ldots, \hat{\theta}^h_{i_{|\mathcal{I}_s|}}), \quad i_k \in \mathcal{I}_s,\quad \hat{\theta}^h_M := \max_{j \in \mathcal{I}_s} \hat{\theta}^h_j, \quad \hat{\theta}^h_m := \min_{j \in \mathcal{I}_s} \hat{\theta}^h_j, \\
&D(\hat{\Theta}^h_s) := \max_{i,j \in \mathcal{I}_s} |\hat{\theta}^h_i - \hat{\theta}^h_j|= \hat{\theta}^h_M - \hat{\theta}^h_m.
\end{aligned}
\right.
\end{align}
Then, we can apply, \eqref{D27}, \eqref{D28}. \eqref{D29} and  Lemma \ref{L2.5} to have the following lemma without proof.
\begin{lemma}\label{L4.4}
For $N\geq 3$, we let $\Theta^h(n)=(\theta^h_1(n), \theta^h_2(n), \ldots, \theta^h_N(n))$ be a solution to the discrete identical Kuramoto model \eqref{A-2} with initial data $\Theta_0\in\mathcal{A}_2$. Then for any given $\varepsilon > 0$ and sufficiently small step size $h\ll 1$, we can find a positive integer $l$ such that
\[|\hat{\theta}^h_N(l)-\frac{N-1}{N}\pi|<\frac{\varepsilon}{2},\quad D(\hat{\Theta}^h_s(l)) < \varepsilon,\quad |\hat{\theta}^h_i(l)+\frac{1}{N}\pi|<\frac{\varepsilon}{2},\quad i=1,\cdots, N-1,\] 
where $N\in\mathcal{I}_b$, $\mathcal{A}_2$ and $\hat{\Theta}^h_s$ are defined in \eqref{D1} and \eqref{D29} respectively.\newline
\end{lemma}

\subsubsection{(The synchronization group $\mathcal{I}_s$)}

We first study the oscillators in the set $\mathcal{I}_s$. 
Then, similar as in Case $\mathcal{A}_1$, we can set $l$ as the initial step and study the large time behavior after $l$. In fact, we have the following result which is almost the same as Lemma \ref{L4.2}.
\begin{lemma}\label{L4.5} 
For $N\geq 3$, we let $\hat{\Theta}^h(n) = (\hat{\theta}^h_1(n), \ldots, \hat{\theta}^h_N(n))$ be a solution to the discrete identical Kuramoto model \eqref{hat system} and $\hat{\Theta}^h_s=(\hat{\theta}^h_{1},\ldots, \hat{\theta}^h_{N-1})$. Moreover, we let the initial configuration satisfy the following properties,
\begin{equation}\label{D30}
 D(\hat{\Theta}^h_s(0)) = \hat{\theta}^h_{N-1}(0) - \hat{\theta}^h_1(0) < \varepsilon,\quad \hat{\theta}^h_{N-1}(0) > \hat{\theta}^h_{N-2}(0) > \cdots > \hat{\theta}^h_1(0),
 \end{equation}
where $\varepsilon$ is a sufficient small positive constant. Then we conclude that the effective phase diameter of $\hat{\Theta}^h_s$ is uniform bounded by the same $\varepsilon$ in \eqref{D30} and the order of the effective phases in $\hat{\Theta}^h_s$ will be preserved for all $n$ i.e.
\begin{align*}
\left\{\begin{aligned}
&D(\hat{\Theta}^h_s(n)) < \varepsilon, \qquad n = 0, 1,2,\ldots,\\
&\hat{\theta}^h_{N-1}(n) > \cdots > \hat{\theta}^h_1(n), \qquad n = 0, 1,2,\ldots.
\end{aligned}
\right.
\end{align*}
\end{lemma}
\begin{proof}
\noindent $\bullet$ (Step$1$): If $D(\hat{\Theta}^h_s(n)) < \varepsilon$  holds for all steps $n$, the proof of the first part of Lemma \ref{L4.5} is done. Then we will prove by induction that the order of oscillators in $\hat{\Theta}^h_s$ will be preserved for all steps $n$. In fact, for the initial step, we have
\[\hat{\theta}^h_{N-1}(0) > \hat{\theta}^h_{N-2}(0) > \cdots > \hat{\theta}^h_{1}(0).\]
Now we assume that the order is preserved for step $k$, then we claim that the order will be preserved for step $k+1$. Actually, we can estimate the difference of $\hat{\theta}^h_i(k+1)$ and $\hat{\theta}^h_{i-1}(k+1)$, where $i =2,3, \cdots, N-1$, as follows, 
\begin{align}
\begin{aligned}\label{D31}
&\hat{\theta}^h_{i}(k+1) -\hat{\theta}^h_{i-1}(k+1)\\
& = \hat{\theta}^h_{i}(k) -\hat{\theta}^h_{i-1}(k) + \frac{Kh}{N} \sum^{N}_{j=1} \Big(\sin(\hat{\theta}^h_j(k) - \hat{\theta}^h_i(k))  - \sin(\hat{\theta}^h_j(k) - \hat{\theta}^h_{i-1}(k))\Big).
\end{aligned}
\end{align}
Then, we apply Lemma \ref{sub-additive} and the mean value theorem to estimate the summation part of \eqref{D31} as below,
\begin{align}
\begin{aligned}\label{D32}
& \sum^{N}_{j=1} [\sin(\hat{\theta}^h_j(k) - \hat{\theta}^h_i(k)) - \sin(\hat{\theta}^h_j(k) - \hat{\theta}^h_{i-1}(k))] \\
&\hspace{0.5cm}= \left[ \sin(\hat{\theta}^h_N(k) - \hat{\theta}^h_{i}(k)) -  \sin(\hat{\theta}^h_N(k) - \hat{\theta}^h_{i-1}(k)) \right] \\
&\hspace{0.5cm}+ \sum^{N-1}_{j=1} \left[ \sin(\hat{\theta}^h_j(k) - \hat{\theta}^h_i(k)) - \sin(\hat{\theta}_{j}(k) - \hat{\theta}^h_{i-1}(k)) \right] \\
&\hspace{0.5cm}\ge - (N - 1) \sin(\hat{\theta}^h_i(k) - \hat{\theta}^h_{i-1}(k)) - \cos \hat{\theta}_{i,i-1}^* (\hat{\theta}^h_{i}(k) - \hat{\theta}^h_{i-1}(k)) \\
&\hspace{0.5cm}\ge - N(\hat{\theta}^h_i(k) - \hat{\theta}^h_{i-1}(k)),
\end{aligned}
\end{align}	
where the value of $\theta_{i,i-1}^*$ is a constant between $\hat{\theta}^h_N(k) - \hat{\theta}^h_{i}(k)$ and $\hat{\theta}^h_N(k) - \hat{\theta}^h_{i-1}(k)$, and the last inequality holds because the order of oscillators at step $k$ is preserved. Thus, we combine the above estimates \eqref{D31}, \eqref{D32} and let $h$ sufficiently small to obtain
\begin{equation}\label{D33}
\hat{\theta}^h_{i}(k+1) - \hat{\theta}^h_{i-1}(k+1)\ge  (1 - Kh)(\hat{\theta}^h_{i}(k) -\hat{\theta}^h_{i-1}(k))>0.
\end{equation}
Therefore, it follows by induction that the order of the oscillators in $\hat{\Theta}^h_s$ is preserved for each step $n$.\newline
	
\noindent $\bullet$ (Step$2$): In step $1$, we show the order is preserved if the diameter is uniformly small. In this step, we will prove by contradiction that $D(\hat{\Theta}^h_s(n))<\varepsilon$ for all $n$, and then the proof in step $1$ is also closed. In fact, we assume $D(\hat{\Theta}^h_s(n)) < \varepsilon$  does not hold for all step $n$. Then, same as in Lemma \ref{L4.2}, there exists a step $n_0$ such that,
\begin{equation}\label{D34}
\left\{
\begin{aligned}
&D(\hat{\Theta}^h_s(n)) < \varepsilon\ll 1, \quad 0 \le n \le n_0,\\
&D(\hat{\Theta}^h_s(n_0+1))  \ge \varepsilon.
\end{aligned}
\right.
\end{equation} 
Using the same argument as in Step$1$, we can obtain that, for each step $0 \leq n \leq n_0 + 1$, the order is preserved, i.e.,
\begin{equation}\label{D35}
\hat{\theta}^h_{i}(n) > \hat{\theta}^h_{i-1}(n), \qquad i = 2,3, \cdots, N-1, \quad 0 \leq n \leq n_0 + 1.
\end{equation}
Now, according to the discrete iteration scheme, we can express the phase diameter $D(\hat{\Theta}^h_s(n))$ at $(n_0 +1)$-th step as below,
\begin{align}
\begin{aligned}\label{D36}
 D(\hat{\Theta}^h_s(n_0 + 1)) &= \hat{\theta}^h_{N-1}(n_0) - \hat{\theta}^h_{1}(n_0) \\
&+ \frac{Kh}{N} \sum^{N}_{j=1} \sin(\hat{\theta}^h_{j}(n_0) - \hat{\theta}^h_{N-1}(n_0))  - \frac{Kh}{N} \sum^{N}_{j=1} \sin(\hat{\theta}^h_{j}(n_0) - \hat{\theta}^h_{1}(n_0)).
\end{aligned}
\end{align}
According to \eqref{D34} and \eqref{D35}, the terms $(\hat{\theta}^h_{j}(n_0) - \hat{\theta}^h_{N-1}(n_0))$ is negative and closed to zero, while $(\hat{\theta}^h_{j}(n_0) - \hat{\theta}^h_{1}(n_0))$ are positive and closed to zero. Therefore, we can apply Lemma \ref{sub-additive} and mean value theorem to obtain,
\begin{align}
\begin{aligned}\label{D37}
&\sum_{j=1}^{N} [\sin(\hat{\theta}^h_j(n_0) - \hat{\theta}^h_{N-1}(n_0)) - \sin(\hat{\theta}^h_j(n_0) - \hat{\theta}^h_{1}(n_0))] \\
&= \sum_{j=1}^{N-1} [\sin(\hat{\theta}^h_j(n_0) - \hat{\theta}^h_{N-1}(n_0)) - \sin(\hat{\theta}^h_j(n_0) - \hat{\theta}^h_{1}(n_0))] \\
&\hspace{0.5cm} + \sin(\hat{\theta}^h_N(n_0) - \hat{\theta}^h_{N-1}(n_0)) - \sin(\hat{\theta}^h_N(n_0) - \hat{\theta}^h_{1}(n_0)) \\
&\le -(N-1) \sin (\hat{\theta}^h_{N-1}(n_0) - \hat{\theta}^h_{1}(n_0)) - \cos \theta_{N-1,1}^* (\hat{\theta}^h_{N-1}(n_0) - \hat{\theta}^h_{1}(n_0)) \\
&\le -(N-1) \frac{\sin \varepsilon}{\varepsilon}(\hat{\theta}^h_{N-1}(n_0) - \hat{\theta}^h_{1}(n_0)) + (\hat{\theta}^h_{N-1}(n_0) - \hat{\theta}^h_{1}(n_0)),
\end{aligned}
\end{align}
where the last inequality follows from the fact that $\frac{\sin x}{x} $ is monotonically decreasing in $[0, \varepsilon]$ when $\varepsilon$ is sufficiently small, and $\theta_{N-1,1}^*$ is a constant between $\hat{\theta}^h_N(n_0) - \hat{\theta}^h_{N-1}(n_0)$ and $\hat{\theta}^h_N(n_0) - \hat{\theta}^h_{1}(n_0)$.
Therefore, we combine \eqref{D36}, \eqref{D37} and the fact $N\geq 3$ to obtain that 

\begin{equation}\label{D38}
\hat{\theta}^h_{N-1}(n_0 +1) - \hat{\theta}^h_{1}(n_0 +1) \le \left[1 - \frac{Kh}{N} \left((N-1) \frac{\sin \varepsilon}{\varepsilon} -1\right) \right] (\hat{\theta}^h_{N-1}(n_0) - \hat{\theta}^h_{1}(n_0)).
\end{equation}
As $\lim_{\varepsilon \to 0 } \frac{\sin \varepsilon}{\varepsilon} = 1$ and $\varepsilon$ is sufficiently small, we immediately conclude from \eqref{D38} that for sufficiently small $h$,
\[D(\hat{\Theta}^h_s(n_0+1))=\hat{\theta}^h_{N-1}(n_0 +1) - \hat{\theta}^h_{1}(n_0 +1)\le \hat{\theta}^h_{N-1}(n_0) - \hat{\theta}^h_{1}(n_0)<\varepsilon,\]
which is a contradiction to $\eqref{D34}_2$. Therefore, $D(\hat{\Theta}^h_s(n)) < \varepsilon$ for all steps $n$ and we finish the proof of the lemma.
\end{proof}

Next, we study the asymptotic synchronization behaviors of oscillators in $\mathcal{I}_s$. The following result states that the convergence to zero of phase diameter $D(\hat{\Theta}^h_s(n))$ is at least exponential.
\begin{lemma}
\label{L4.6}
For $N\geq 3$, we let $\hat{\Theta}^h(n) = (\hat{\theta}^h_1(n), \ldots, \hat{\theta}^h_N(n))$ be a solution to the discrete identical Kuramoto model \eqref{hat system} and $\hat{\Theta}^h_s=(\hat{\theta}^h_{1},\ldots, \hat{\theta}^h_{N-1})$. Moreover, we assume that the initial configuration satisfies the following conditions:
\[ \hat{\theta}^h_{N-1}(0) > \hat{\theta}^h_{N-2}(0) > \cdots > \hat{\theta}^h_1(0), \qquad D(\hat{\Theta}^h_s(0)) < \varepsilon,\]
where $\varepsilon$ is a sufficiently small positive number. Then, the diameter $D(\hat{\Theta}^h_s(n))$ is strictly monotonically decreasing, and moreover there exist positive numbers $h_0 $ and $\alpha$ such that, for $0<h<h_0$,
\[D(\hat{\Theta}^h_s(0)) \exp \left\{ - 2K nh\right\}<D(\hat{\Theta}^h_s(n)) < D(\hat{\Theta}^h_s(0)) \exp \left\{ - \alpha nh\right\}, \quad n \geq 0.\]
\end{lemma}
\begin{proof}
According to Lemma \ref{L4.5}, we know the order and diameter are both preserved for $\hat{\Theta}^h_s(n)$. Moreover, we have the estimate \eqref{D38}. Then, the iteration of \eqref{D38} leads to the estimates of the diameter of $\hat{\Theta}^h_s(n)$ as below,
\begin{equation*}
D(\hat{\Theta}^h_s(n+1))< \left[1 - \frac{Kh}{N}  \left((N-1) \frac{\sin \varepsilon}{\varepsilon} -1\right) \right]^{n+1} D(\hat{\Theta}^h_s(0)),
\end{equation*}
where $N\geq 3$. As $\varepsilon$ and $h$ are sufficiently small, we can follow the proof of Lemma \ref{L4.3} and apply L'Hospital's rule to obtain that 	
%	\begin{align*}
%		\begin{aligned}
%		D(\hat{\Theta}^h_s(n)) &< \left[1 - \frac{Kh}{N} \left((N-1) (1- \delta) -1\right) \right]^{n} D(\hat{\Theta}^h_s(0)) \\
%		&= \exp \left\{\log \left[1 - \frac{Kh}{N} \left((N-1) (1- \delta) -1\right) \right]^{n} \right\} D(\hat{\Theta}^h_s(0)) \\
%		&= \exp \left\{nh \frac{\log \left[1 - \frac{Kh}{N} \left((N-1) (1- \delta) -1\right) \right]}{h} \right\} D(\hat{\Theta}^h_s(0)) .
%		\end{aligned}
%		\end{align*}
%		However, we exploit L'Hospital's rule to obtain
%		\[\lim_{h \to 0} \frac{\log \left[1 - \frac{Kh}{N} \left((N-1) (1- \delta) -1\right) \right]}{h} = - \frac{K[(N-1)(1-\delta) -1]}{N}.\]
%		Therefore, there exists $h_0>0$ such that if $0<h<h_0$, we have
%		\[\frac{\log \left[1 - \frac{Kh}{N} \left((N-1) (1- \delta) -1\right) \right]}{h} < - \frac{K[(N-1)(1-\delta) -1]}{2N}.\]
%		It yields that
\begin{equation}\label{D39}
D(\hat{\Theta}^h_s(n)) < D(\hat{\Theta}^h_s(0)) \exp \left\{ - \alpha nh\right\}, \quad \alpha=\frac{K[(N-1) \frac{\sin \varepsilon}{\varepsilon} -1]}{2N}>0, \quad n\geq 0.
\end{equation}
On the other hand, according to Lemma \ref{L4.5}, we can apply similar analysis in \eqref{D37} to obtain the following estimate,
\begin{align}
\begin{aligned}\label{D40}
 &D(\hat{\Theta}^h_s(n + 1)) \\
 & = \hat{\theta}^h_{N-1}(n) - \hat{\theta}^h_{1}(n)+ \frac{Kh}{N} \sum^{N}_{j=1} \sin(\hat{\theta}^h_{j}(n) - \hat{\theta}^h_{N-1}(n))  - \frac{Kh}{N} \sum^{N}_{j=1} \sin(\hat{\theta}^h_{j}(n) - \hat{\theta}^h_{1}(n))\\
 &\geq  \hat{\theta}^h_{N-1}(n) - \hat{\theta}^h_{1}(n)-Kh(\hat{\theta}^h_{N-1}(n) - \hat{\theta}^h_{1}(n))\\
 &=(1-Kh)D(\hat{\Theta}^h_s(n)).
\end{aligned}
\end{align}
Then, following again the proof of Lemma \ref{L4.3}, we apply \eqref{D40} and L'Hospital's rule to obtain that
\begin{equation}\label{D41}
D(\hat{\Theta}^h_s(n + 1))\geq \exp\{-2K(n+1)h\}D(\hat{\Theta}^h_s(0)).
\end{equation}
Finally, we combine \eqref{D39}, \eqref{D40} and \eqref{D41} to finish the proof of the present lemma. 
\end{proof}
\begin{remark}\label{R4.3}
The results in this part only show the asymptotical synchronization of $\hat{\Theta}^h_s$, but we still do not know which equilibrium state does $\hat{\Theta}^h_s$ approach. To study the asymptotical equilibrium state, we have to use the conservation of the mean phase of all oscillators. Therefore, we need the information of the $N$-th oscillator.\newline  
\end{remark}

\subsubsection{(The behavior of $N$-th oscillator)}
According to Remark \ref{R4.3}, in order to understand the asymptotical behavior of all oscillators, we have to study the behavior of $N$-th oscillator. For the continuous Kuramoto model, $\Theta(t)$ will approach bipolar formation if $\Theta_0\in\mathcal{A}_2$. However, for the same initial data $\Theta_0\in\mathcal{A}_2$, we cannot guarantee the emergence of bipolar formation in discrete-time Kuramoto model, which may be due to the instability of the bipolar formation in the continuous model. More precisely, according to \eqref{D27} and the definition of the efficient phases, we know that $\theta_i$ are closed to $-\frac{\pi}{N}$ for $i\in\mathcal{I}_s$ and $\theta_N$ is closed to $\frac{N-1}{N}\pi$ at $T_{\varepsilon}$ after a translation. Therefore, as bipolar formation emerges asymptotically for $\Theta_0\in\mathcal{A}_2$, we immediately have 
\begin{equation}\label{D42}
\theta_1(t)+\pi<\theta_N(t)<\theta_{N-1}(t)+\pi,\quad \Theta_0\in\mathcal{A}_2\quad t\geq T_{\varepsilon}.
\end{equation}
Otherwise, all the particles will be contained in a half circle at some time $t\geq T_{\varepsilon}$, and thus the complete synchronization will emerge asymptotically which is a well known result for continuous identical Kuramoto model. However, as $\Theta^h(n)$ and $\Theta(nh)$ have nonzero error, we cannot tell if \eqref{D42} holds for $\Theta^h$ at step $l$, where $l=\frac{T_{\varepsilon}}{h}$ and $\Theta^h$ is the solution to the discrete-time identical Kuramoto model with initial data $\Theta_0\in\mathcal{A}_2$. Therefore, we will study the large time behavior of discrete-time Kuramoto model with initial data $\Theta_0\in\mathcal{A}_2$ in two different cases.\newline   
	
\noindent $\diamond$ Case i: (bipolar emergence) According to Lemma \ref{L4.4}, $\hat{\Theta}^h_s$ are closed to $-\frac{1}{N}\pi$ for $i\in\mathcal{I}_s$ and $\hat{\theta}^h_N$ is closed to $\frac{N-1}{N}\pi$ at step $l$. Let's suppose \eqref{D42} holds for $\hat{\Theta}^h(n)$ at any step $n\geq l$, i.e. we assume 
\begin{equation}\label{D43}
\hat{\theta}^h_1(n) + \pi \le \hat{\theta}^h_N(n) \le \hat{\theta}^h_{N-1}(n) + \pi, \qquad n \ge l.
\end{equation}
%	always holds. Another case is 
%\[ \hat{\theta}^h_1(n) - \pi \le \hat{\theta}^h_N(n) \le \hat{\theta}^h_{N-1}(n) - \pi, \qquad n \ge l,\]
%always holds. We will only discuss the case \eqref{D43}, the other case can be treated similarly.
\begin{lemma}\label{L4.7}
For $N\geq 3$, let $\hat{\Theta}^h(n) = (\hat{\theta}^h_1(n), \ldots, \hat{\theta}^h_N(n))$ be a solution to the discrete identical Kuramoto model \eqref{hat system} and $\hat{\Theta}^h_s=(\hat{\theta}^h_{1},\ldots, \hat{\theta}^h_{N-1})$.
Moreover, we assume that 
\begin{equation}\label{D44}
\left\{
\begin{aligned}
&\hat{\theta}^h_{N-1}(0) > \hat{\theta}^h_{N-2}(0) > \cdots > \hat{\theta}^h_1(0),\\
& D(\hat{\Theta}^h_s(0)) < \varepsilon,\quad |\hat{\theta}^h_N(0)- \frac{(N-1)\pi}{N}|<\frac{\varepsilon}{4},\\
&\hat{\theta}^h_1(n) + \pi \le \hat{\theta}^h_N(n) \le \hat{\theta}^h_{N-1}(n) + \pi, \qquad n \ge 0.
\end{aligned}
\right.
\end{equation}
Then, for sufficiently small time-step $h$, we have
\begin{align*}
\left\{\begin{aligned}
&|\hat{\theta}^h_N(n) - \frac{(N-1)\pi}{N}| < \frac{N-1}{N}D(\hat{\Theta}^h_s(0))e^{- \alpha nh},\\
&|\hat{\theta}^h_j(n) + \frac{\pi}{N}| < \frac{2N-1}{N}D(\hat{\Theta}^h_s(0))e^{- \alpha nh}, \quad j \in \mathcal{I}_s.
\end{aligned}
\right.
\end{align*}
\end{lemma}
\begin{proof}
According to \eqref{D27} and \eqref{D28}, the sum of effective phases is equal to zero. Therefore,  it is clear that 
\begin{equation}\label{D45}
\frac{\sum_{i=1}^{N-1} (\hat{\theta}^h_i(n) + \pi) + \hat{\theta}^h_N(n)}{N} = \frac{(N-1)\pi}{N}.
\end{equation}
As the initial data satisfies \eqref{D44}, according to Lemma \ref{L4.6} and the condition $\eqref{D44}_3$, we have for $j \in \mathcal{I}_s$ that,
\begin{align}\label{D46}
\begin{aligned}
\left| \hat{\theta}^h_N(n) - (\hat{\theta}^h_j(n) + \pi) \right| &< (\hat{\theta}^h_{N-1}(n) + \pi) - (\hat{\theta}^h_1(n) + \pi) \le e^{- \alpha nh} D(\hat{\Theta}^h_s(0)).
\end{aligned}
\end{align}
Combining \eqref{D45} and \eqref{D46}, we obtain that
\begin{align*}
\begin{aligned}
|\hat{\theta}^h_N(n) - \frac{N-1}{N}\pi| &=|\hat{\theta}^h_N(n) - \frac{\sum_{i=1}^{N-1} (\hat{\theta}^h_i(n) + \pi) + \hat{\theta}^h_N(n)}{N}|\\
& = |\frac{\sum_{i=1}^{N-1} [\hat{\theta}^h_N(n) - (\hat{\theta}^h_i(n)+\pi)]}{N}| < \frac{N-1}{N} D(\hat{\Theta}^h_s(0)) e^{- \alpha nh}.
\end{aligned}
\end{align*}
Finally, for $i \in \mathcal{I}_s$, we have
\begin{align*}
\begin{aligned}
|\hat{\theta}^h_i(n) + \frac{\pi}{N}| &= |\hat{\theta}^h_i(n) - (\hat{\theta}^h_N(n)-\pi) + (\hat{\theta}^h_N(n)-\pi) + \frac{\pi}{N}| \\
&\le |(\hat{\theta}^h_i(n) + \pi)- \hat{\theta}^h_N(n)| + |\hat{\theta}^h_N(n) - \frac{N-1}{N}\pi| < \frac{2N-1}{N}D(\hat{\Theta}^h_s(0))e^{- \alpha nh}.
\end{aligned}
\end{align*}
%	\[\lim_{n \to +\infty} \left( \theta_i^h(n) - \theta_j^h(n) \right) =0, \quad i,j \in \mathcal{I}_s,\]
%	\[\lim_{n \to +\infty} \left(\theta_i^h(n) + \pi - (\theta_j^h(n) + \pi) \right) = 0, \quad i,j \in \mathcal{I}_s,\]
%	\[\lim_{n \to +\infty} (\theta_N(n) - (\theta_j^h(n) + \pi)) = 0, \quad j \in \mathcal{I}_s,\]
%	\[\sum_{j=1}^{N-1} \lim_{n \to +\infty} (\theta_N(n) - (\theta_j^h(n) + \pi)) = \lim_{n \to +\infty} [(N-1) \theta_N(n) - \sum_{j=1}^{N-1} (\theta_j^h(n)+\pi) ]= 0\]
%	\[\lim_{n \to +\infty} N\theta_N(n) = (N-1)\pi \quad \Longrightarrow \quad \lim_{n \to +\infty} \theta_N(n) = \frac{N-1}{N}\pi\]
%	\[\lim_{n \to +\infty} \theta_N(n) - \lim_{n \to +\infty} (\theta_N(n) - (\theta_j^h(n) + \pi)) = \frac{N-1}{N}\pi \quad \Longrightarrow \quad \lim_{n \to +\infty} \theta_j^h(n) = -\frac{1}{N} \pi, \quad j \in \mathcal{I}_s^d\]
%	Furthermore, it follows that
%	\[ \lim_{n \to +\infty} \hat{\theta}^h_N(n) = \frac{N-1}{N}\pi, \qquad \lim_{n \to +\infty} \hat{\theta}^h_j(n) = -\frac{1}{N} \pi, \quad j \in \mathcal{I}_s.\]
\end{proof}
\begin{remark}\label{R4.4}
Owing to \eqref{D27} and \eqref{D28}, we obtain from Lemma \ref{L4.7} for (Case i) that
\begin{equation*}
\left\{
\begin{aligned}
&\left|\theta^h_N(n) - (2k_N+1)\pi - \phi_1^* \right| < \frac{N-1}{N}D(\hat{\Theta}^h_s(0))e^{- \alpha nh},\\
&|\theta_j^h(n) -2k_j\pi - \phi_1^*| < \frac{2N-1}{N}D(\hat{\Theta}^h_s(0))e^{- \alpha nh}, \quad j \in \mathcal{I}_s.
\end{aligned}
\right.
\end{equation*}
\end{remark}
\vspace{0.5cm}

\noindent $\diamond$ Case ii: (phase synchronization) According to Lemma \ref{L4.7}, we know the bipolar formation will emerge if \eqref{D44} holds. Then, the second case is that $\eqref{D44}_3$ does not hold for some step. In other words, there exists a step $n_e\geq l$ such that, $\hat{\theta}^h_N$ gets out of the region $(\hat{\theta}^h_{1}(n)+\pi,\hat{\theta}^h_{N-1}(n) + \pi)$ first time at step $n_e$. Then, as the step size $h$ is sufficiently small, it's obvious that $\hat{\theta}^h_N(n_e)$ is either slightly smaller than $\hat{\theta}^h_{1}(n_e)+\pi$ or slightly greater than $\hat{\theta}^h_{N-1}(n_e)+\pi$. As these two cases can be analyzed similarly, we will only study the following case, 
\begin{equation}
\left\{\begin{aligned}\label{D47}
&\hat{\theta}^h_{N-1}(n_e) < \hat{\theta}^h_N(n_e) < \hat{\theta}^h_1(n_e) + \pi,\\
&\hat{\theta}^h_{1}(n)+\pi < \hat{\theta}^h_N(n) < \hat{\theta}^h_{N-1}(n) + \pi,\quad l\leq n<n_e.
\end{aligned}
\right.
\end{equation}
\begin{remark}\label{R4.5}
It's possible that $\hat{\theta}^h_N(n_e)=\hat{\theta}^h_1(n_e) + \pi$ at step $n_e$. However, according to Lemma \ref{L4.6}, we know the diameter $D(\hat{\Theta}^h_s)$ will be nonzero at any finite step if it is nonzero initially. Therefore, there must be some oscillator $\hat{\theta}^h_i$ such that $\hat{\theta}^h_1(n_e)<\hat{\theta}^h_i(n_e)<\hat{\theta}^h_N$. Then, if $\hat{\theta}^h_N(n_e)=\hat{\theta}^h_1(n_e) + \pi$ at step $n_e$, the attraction from $\hat{\theta}^h_i$ to $\hat{\theta}^h_1$ and $\hat{\theta}^h_N$ will force them tend to closer in the next step, i.e.
\[\hat{\theta}^h_N(n_e+1)<\hat{\theta}^h_1(n_e+1) + \pi.\]
Therefore, we only need to study the case \eqref{D47}.
\end{remark}
\noindent According to \eqref{D47} and the Remark \ref{R4.5}, we can find a positive constant $\eta$ such that the oscillators $\hat{\Theta}^h$ satisfy the following properties at step $n_e$, 
\begin{equation}\label{D48}
 \hat{\theta}^h_1(n_e) < \cdots < \hat{\theta}^h_{N-1}(n_e) < \hat{\theta}^h_{N}(n_e), \quad \hat{\theta}^h_{N}(n_e) - \hat{\theta}^h_1(n_e) <\eta <  \pi.
 \end{equation}
Now, we can set $n_e$ as the initial data and study the large time behavior of the discrete-time model after the step $n_e$.
\begin{lemma}\label{uniform boundness of diameter and order of N oscillators when |I_b| is 1 and  N-1 less N less 1 plus pi} 
For $N\geq 3$, we let $\hat{\Theta}^h(n) = (\hat{\theta}^h_1(n), \ldots, \hat{\theta}^h_N(n))$ be a solution to the discrete identical Kuramoto model \eqref{hat system} with initial zero total phase $\sum_{i=1}^N \hat{\theta}_i^0 = 0$. Moreover, we assume that the initial configuration satisfies the following conditions:
\[ \hat{\theta}^h_1(0) < \cdots < \hat{\theta}^h_{N-1}(0) < \hat{\theta}^h_{N}(0), \qquad D(\hat{\Theta}^h(0)) < \eta.\]
Then for $n\geq 0$, we have
\begin{equation*}
\left\{
\begin{aligned}
&D(\hat{\Theta}^h(n)) = \max_{1\leq i,j \leq N} |\hat{\theta}^h_i(n) - \hat{\theta}^h_j(n)| < \eta < \pi,\\
&\hat{\theta}^h_1(n) < \cdots < \hat{\theta}^h_{N-1}(n) < \hat{\theta}^h_{N}(n).
\end{aligned}
\right.
\end{equation*}
\end{lemma}
\begin{proof}
If $D(\hat{\Theta}^h(n)) < \eta < \pi$  holds for all step $n$, we can apply the same argument in Lemma \ref{L4.5} to prove that the order of oscillators $D(\hat{\Theta}^h(n))$ is preserved for any step $n$. Therefore, we only need to verify the inequality $D(\hat{\Theta}^h(n)) < \eta < \pi$ for $n\geq 0$. Actually, suppose not, then there exists some step $n_0$ such that
\begin{equation}\label{D49}
\left\{\begin{aligned}
& D(\hat{\Theta}^h(n)) < \eta < \pi,\quad n\leq n_0,\\
& D(\hat{\Theta}^h(n_0+1)) \ge \eta. 
\end{aligned}
\right.
\end{equation}
Then, similar as in Lemma \ref{L4.5}, we can derive the order of $\hat{\Theta}^h(n)$ for time steps not more than $n_0+1$, i.e.
\begin{equation}\label{D50}
\hat{\theta}^h_1n) < \hat{\theta}^h_{2}(n) < \cdots < \hat{\theta}^h_N(n), \quad 0\leq n\leq n_0+1.
\end{equation}
Then, applying Lemma \ref{sub-additive} and monotonically decresing of function $\frac{\sin x}{x}$ in $(0, \pi)$, we obtain the estimate for the phase diameter at the $(n_0 +1)$-th step as follows,
\begin{align}
\begin{aligned}\label{D51}
&D(\hat{\Theta}^h(n_0 + 1)) \\
&= \hat{\theta}^h_{N}(n_0) - \hat{\theta}^h_{1}(n_0) + \frac{Kh}{N} \sum^{N}_{j=1} \sin(\hat{\theta}^h_{j}(n_0) - \hat{\theta}^h_{N}(n_0))- \frac{Kh}{N} \sum^{N}_{j=1} \sin(\hat{\theta}^h_{j}(n_0) - \hat{\theta}^h_{1}(n_0)) \\
&\le \left( 1- Kh \frac{\sin \eta}{\eta} \right)  (\hat{\theta}^h_{N}(n_0) - \hat{\theta}^h_{1}(n_0))\\
&\leq\hat{\theta}^h_{N}(n_0) - \hat{\theta}^h_{1}(n_0),
\end{aligned}
\end{align}
where the last inequality holds if we choose sufficiently small $h$. Then $\eqref{D49}_1$ and \eqref{D51} immediately imply that $D(\hat{\Theta}^h(n_0 + 1))<\eta$, which is obviously a contradiction to $\eqref{D49}_2$. Therefore, we conclude $D(\hat{\Theta}^h(n)) < \eta < \pi$ for $n\geq 0$, and thus the order \eqref{D50} holds for all $n\geq 0$.
\end{proof}

\noindent Next, we study the asymptotic synchronization behaviors of  N oscillators in (Case ii). The following result states that the effective phases $\hat{\Theta}^h(n)$ will converge to zero exponentially in (Case ii).

\begin{lemma}\label{L4.9}
For $N\geq 3$, we let $\hat{\Theta}^h(n) = (\hat{\theta}^h_1(n), \ldots, \hat{\theta}^h_N(n))$ be a solution to the discrete identical Kuramoto model \eqref{hat system} with initial zero total phase $\sum_{i=1}^N \hat{\theta}_i^0 = 0.$ Moreover, we assume that the initial configuration satisfies the following conditions:
\[ \hat{\theta}^h_1(0) < \cdots < \hat{\theta}^h_{N-1}(0) < \hat{\theta}^h_{N}(0), \qquad D(\hat{\Theta}^h(0)) < \eta.\]
Then, there exists a positive constant $h_0$ such that for $0<h<h_0$,
\begin{align*}
\left\{\begin{aligned}
&D(\hat{\Theta}^h(n)) < D(\hat{\Theta}(0)) \exp\left(-\frac{K\sin \eta}{2\eta} nh \right),\\ 
&|\hat{\theta}^h_j(n)| < D(\hat{\Theta}(0)) \exp\left(-\frac{K\sin \eta}{2\eta} nh \right),\quad n\geq 0.
\end{aligned}
\right.
\end{align*}
%\[D(\hat{\Theta}^h(n)) < D(\hat{\Theta}(0)) \exp\left(-\frac{K\sin \eta}{2\eta} nh \right), \qquad n = 0,1,2,\ldots, \]
%		and
%		\[|\hat{\theta}^h_j(n)| < D(\hat{\Theta}(0)) \exp\left(-\frac{K\sin \eta}{2\eta} nh \right), \qquad j = 1,2, \ldots, N.\]
%		Moreover,
%		\[\lim_{n \to +\infty} \hat{\theta}^h_j(n) = 0, \quad j = 1,2, \ldots, N.\]
\end{lemma}
\begin{proof}
From Lemma \ref{uniform boundness of diameter and order of N oscillators when |I_b| is 1 and  N-1 less N less 1 plus pi} and the estimate \eqref{D51}, we can apply the same argument in Lemma \ref{L4.3} to obtain the desired results. 

%that estimate any step phase diameter,
%		\begin{align*}
%		\begin{aligned}
%		D(\hat{\Theta}^h(n+1)) &= \hat{\theta}_{N}(n +1) - \hat{\theta}_{1}(n + 1) \\
%		&\le \left(1 - Kh \frac{\sin \eta}{\eta} \right)^{n+1}(\hat{\theta}_{N}(0) - \hat{\theta}_{1}(0)) \\
%		&= \left(1 - Kh \frac{\sin \eta}{\eta} \right)^{n+1} D(\hat{\Theta}(0))
%		\end{aligned}
%		\end{align*}
%		where we use the method of iteration.
%		Then, we have
%		\begin{align*}
%		\begin{aligned}
%		D(\hat{\Theta}^h(n)) &\le \left(1 - Kh \frac{\sin \eta}{\eta} \right)^{n} D(\hat{\Theta}(0)) \\
%		&= \exp \left[ \log\left(1 - Kh \frac{\sin \eta}{\eta} \right)^{n} \right] D(\hat{\Theta}(0)) \\
%		&= \exp \left[ nh \frac{\log\left(1 - Kh \frac{\sin \eta}{\eta} \right)}{h} \right] D(\hat{\Theta}(0)).
%%		&\le e^{- \alpha_3 nh} D(\Theta(0)),
%		\end{aligned}
%		\end{align*}
%		However, we exploit L'Hospital's rule to obtain
%		\[\lim_{h \to 0} \frac{\log \left(1 - Kh \frac{\sin \eta}{\eta} \right)}{h} = -\frac{K\sin \eta}{\eta}.\]
%		Therefore, there exists $h_0>0$ such that if $0<h<h_0$, we have
%		\[\frac{\log \left(1 - Kh \frac{\sin \eta}{\eta} \right)}{h} < -\frac{K\sin \eta}{2\eta}.\]
%		It yields that
%		\[D(\hat{\Theta}^h(n)) < D(\hat{\Theta}(0)) \exp\left(-\frac{K\sin \eta}{2\eta} nh \right), \qquad n = 0,1,2,\ldots.\]
%	Similarly, we R4.6can use the same calculations in Lemma \ref{L4.3} to obtain
%	\[|\hat{\theta}^h_j(n)| < D(\hat{\Theta}(0)) \exp\left(-\frac{K\sin \eta}{2\eta} nh \right), \qquad j = 1,2, \ldots, N,\]
%	and
%		\[\lim_{n \to +\infty} \hat{\theta}^h_j(n) = 0, \qquad j = 1,2, \ldots, N.\]
\end{proof}
\begin{remark}\label{R4.6}
Owing to \eqref{D27} and \eqref{D28}, we obtain from Lemma \ref{L4.9} for (Case ii) that
\[|\theta_i^h(n) - 2k_i\pi - \phi_1^* - \frac{\pi}{N}| < D(\hat{\Theta}(0)) \exp\left(-\frac{K\sin \eta}{2\eta} nh \right), \quad i = 1,2,\ldots, N.\]\newline
%Due to the relation
%\[\hat{\theta}^h_i(n) = \theta_i^h(n) - 2k_i\pi - \phi_1^* - \frac{\pi}{N},\]
%it follows from Lemma \ref{L4.9} that
%\[|\theta_i^h(n) - 2k_i\pi - \phi_1^* - \frac{\pi}{N}| < D(\hat{\Theta}(0)) \exp\left(-\frac{K\sin \eta}{2\eta} nh \right), \quad i = 1,2,\ldots, N,\]
%and
%\[\lim_{n \to +\infty} \theta_i^h(n) = 2k_i \pi + \phi_1^* + \frac{\pi}{N}, \quad i = 1,2,\ldots,N,\]
%which is consistent with Lemma \ref{priori estimate}.
\end{remark}

\subsection{Proof of Theorem \ref{T4.1}} Now, we are ready to prove Theorem \ref{T4.1}. Actually, combining Remark \ref{R4.2}
and Remark \ref{R4.4}, we directly conclude that, for any initial data satisfying the condition in Theorem \ref{T4.1}, we can find a time step $n_e$ and an equilibrium state $\Theta^\infty$ such that 
\[\| \Theta^h(n) - \Theta^\infty\|_\infty < Ce^{-\alpha (n-n_e)h}, \qquad n\ge n_e,\]
where $C$, $\alpha$ and $n_e$ are positive constants depending on initial data. Moreover, according to the expression of $\phi_0^*$ and $\phi_1^*$, the form of phase locked state in Remark \ref{R4.6} are equivalent to the equilibrium state constructed in Remark \ref{R4.2}. Therefore, we only have two types of phase locked states as mentioned in Theorem \ref{T4.1}, i.e.
\begin{enumerate}
\item $\Theta^\infty =(2k_1 \pi + \phi_0^*,\ldots, 2k_{N}\pi + \phi_0^*)$, or 
\item $\Theta^\infty = (2k_1\pi + \phi_1^*, \ldots, 2k_{N-1} \pi + \phi_1^*, (2k_N+1)\pi + \phi_1^*)$,
\end{enumerate}
where $k_i \in \bbz, \ i =1,2,\ldots,N$, which finish the proof of the theorem.\newline 
\qed 

\vspace{0.5cm}
\section{Discrete nonidentical Kuramoto model}\label{sec:5}
\setcounter{equation}{0}
\vspace{0.5cm}
In this section, we will consider the discrete non-identical Kuramoto model \eqref{A-2}. Actually, from the analysis in \cite{H-K-K-Z}, if the initial data is contained in a quarter, then the synchronization of the oscillators in discrete nonidentical Kuramoto model will be guaranteed for sufficiently large coupling strength $K$. If the initial data is contained in a half circle, Lemma \ref{L2.3} shows the oscillators in continuous nonidentical Kuramoto model will concentrate into a quarter after finite time for sufficiently large coupling strength. Therefore, combining Lemma \ref{L2.5} and the analysis in \cite{H-K-K-Z}, we can conclude the emergence of synchronization in nonidentical case for sufficiently large coupling strength. Moreover, if the initial data $\Theta_0\in\mathcal{A}_1$, according to \cite{H-K-R}, both identical and nonidentical oscillators will concentrate into a small region after finite time, provided the coupling strength is sufficiently large. Therefore, we can again apply Lemma \ref{L2.5} and the analysis in \cite{H-K-K-Z} to obtain the emergence of synchronization.\newline 

However, for initial data $\Theta_0\in\mathcal{A}_2$, the oscillators may not move to the quarter, thus we cannot apply previous analysis to yield the emergence of synchronization. Fortunately, we have Theorem \ref{T3.1} based on the gradient flow structure of discrete model. Therefore, the strategy in \cite{H-K-R} can be extended to the discrete version. More precisely, in order to show the emergence of synchronization, we only need to show the uniformly boundedness of the nonidentical oscillators. In the following lemma, we will provide a sufficient condition for uniform bound of the first $N_0$ nonidentical oscillators.

\begin{lemma}\label{L5.1}
Let $N\geq 3$, suppose that the initial configuration $\Theta(0)$ and natural frequencies satisfy the following conditions,
\[\frac{1}{N} \sum_{j=1}^N \theta_j(0) = 0, \quad \frac{1}{N} \sum_{j=1}^N \Omega_j = 0,,\quad \theta_j(0) \in [-\pi, \pi), \quad 1 \le j \le N. \]
Moreover, we let $N_0$, $l$ and $K$ be positive constants which satisfy the following conditions,
\begin{equation}\label{BD-1}
\begin{aligned}
&N_0 \in \bbz^+ \cap \left(\frac{N}{2}, N\right], \qquad l \in \left(0, 2 \arccos \frac{N-N_0}{N_0}\right), \\
& \max_{1 \le j,k \le N_0} |\theta_j(0) - \theta_k(0)| < l, \qquad K > \frac{D(\Omega)}{\frac{N_0}{N} \sin l - \frac{2(N-N_0)}{N} \sin \frac{l}{2}}. 
\end{aligned}
\end{equation}
Then, for the solution $\Theta^h(n)$  to the discrete non-identical system \eqref{A-2}, there exists a positive constant $h_0 $ such that, for $0 < h < h_0$ we have
\begin{equation}\label{BD-2}
\max_{1 \le j,k \le N_0} |\theta_j^h(n) - \theta^h_k(n)| < l, \quad \text{for all } n \ge 0.
\end{equation}
\end{lemma}
\begin{proof}
let $\Theta^h(n) = (\theta^h_1(n), \theta^h_2(n), \ldots, \theta^h_N(n))$ be a solution to the discrete system \eqref{A-2} subject to the initial configuration satisfying the conditions in the statements of theorem. Then $l$ is obviously less than $\pi$, i.e. $0 < l < \pi$, owing to the assumption on $l$ in \eqref{BD-1}. Then, we will prove \eqref{BD-2} by contradiction. In fact, suppose \eqref{BD-2} does not hold, then there exists a step $n_*$ such that 
\begin{equation}\label{BD-3}
\left\{\begin{aligned}
&\max_{1 \le j,k \le N_0} |\theta_j^h(n) - \theta^h_k(n)| < l, \quad 0 \le n\le n_*,\\
&\max_{1 \le j,k \le N_0} |\theta^h_j(n_* + 1) - \theta^h_k(n_* + 1)| \ge l.
\end{aligned}
\right.
\end{equation}
Let $P$, $p$, $Q$ and $q$ be integers in $[1,N_0]$. Without loss of generality, we may set the $P$-th and $p$-th oscillators to be the maximum and minimum of the first $N_0$ oscillators at step $n_*$, respectively. Similarly, we let $Q$-th and $q$-th oscillators  be the maximum and minimum respectively of the first $N_0$ oscillators at step $n_*+1$, i.e.
\begin{equation}\label{E4}
\begin{aligned}
&\theta^h_P(n_*) = \max_{1 \le j \le N_0} \theta^h_j(n_*), \qquad \theta^h_p(n_*) = \min_{1 \le j\le N_0} \theta^h_j(n_*), \\
& \theta^h_Q(n_*+1) = \max_{1 \le j \le N_0} \theta^h_j(n_*+1), \qquad \theta^h_q(n_*+1) = \min_{1 \le j \le N_0} \theta^h_j(n_*+1).
\end{aligned}
\end{equation}
%In the following discussion, without loss of generality assume $M \ne q$ and $m \ne q.$ It is known from \eqref{BD-3} that
%\[\theta_Q(n_*) - \theta_q(n_*) <l, \quad \text{and} \quad \theta_Q(n_*+1) - \theta_q(n_*) \ge l.\]
%\[\theta_Q(n_*+1) = \theta_Q(n_*) + h\Omega_p + \frac{Kh}{N} \sum_{j=1}^N \sin (\theta_j(n_*) -  \theta_Q(n_*)),\]
%\[\theta_q(n_*+1) = \theta_q(n_*) + h\Omega_q + \frac{Kh}{N} \sum_{j=1}^N \sin (\theta_j(n_*) - \theta_q(n_*)).\]
%\[\theta_p(n_*+1) = \theta_p(n_*) + h\Omega_m + \frac{Kh}{N} \sum_{j=1}^N \sin (\theta_j(n_*) - \theta_p(n_*)),\]
%\[\theta_k(n_*+1) = \theta_k(n_*) + h\Omega_k + \frac{Kh}{N} \sum_{j=1}^N \sin (\theta_j(n_*) - \theta_k(n_*)).\]
Then, the phase diameter of the first $N_0$ oscillators at $(n_*+1)$-th step can be expressed by the $n_*$-th step information due to the iteration scheme,
\begin{align}
\begin{aligned}\label{BD-4}
&\theta^h_Q(n_*+1) - \theta^h_q(n_*+1) \\
&= \theta^h_Q(n_*) - \theta^h_q(n_*)  + h(\Omega_p - \Omega_q)  + \frac{Kh}{N} \sum_{j=1}^N [\sin (\theta^h_j(n_*) - \theta^h_Q(n_*)) - \sin (\theta^h_j(n_*) -\theta^h_q(n_*))] \\
&\le \theta^h_Q(n_*) - \theta^h_q(n_*) + hD(\Omega) - \frac{2Kh}{N} \sin \frac{\theta^h_Q(n_*) - \theta^h_q(n_*)}{2} \sum_{j=1}^N \cos \left( \theta^h_j(n_*)-\frac{\theta^h_Q(n_*) + \theta^h_q(n_*)}{2}\right) 
\end{aligned}
\end{align}
Now, we will first show that the last term in \eqref{BD-4} is positive and then, we can apply the estimates on the trigonometric functions to obtain desire results.\newline

\noindent $\bullet$ Step 1. (Positivity) In the last term of \eqref{BD-4}, to deal with the summation of trigonometric functions, we may divide it into two parts as follows,
\begin{align}
\begin{aligned}\label{E6}
&\sum_{j=1}^N \cos \left( \theta^h_j(n_*)-\frac{\theta^h_Q(n_*) + \theta^h_q(n_*)}{2}\right) \\
&=  \sum_{j=N_0+1}^N \cos \left( \theta^h_j(n_*)-\frac{\theta^h_Q(n_*) + \theta^h_q(n_*)}{2}\right)+\sum_{j=1}^{N_0} \cos \left( \theta^h_j(n_*)-\frac{\theta^h_Q(n_*) + \theta^h_q(n_*)}{2}\right) .
\end{aligned}
\end{align}
For the first term, we apply the uniformly boundedness of the trigonometric functions to have the simple lower bound estimates as below,
\begin{equation}\label{E7}
\sum_{j=N_0+1}^N \cos \left( \theta^h_j(n_*)-\frac{\theta^h_Q(n_*) + \theta^h_q(n_*)}{2}\right) \ge -(N-N_0).
\end{equation}
For $1 \le j \le N_0$, according to \eqref{E4}, $\theta_P$ and $\theta_p$ are the maximum and minimum at $n_*$-th step. Therefore, we can apply \eqref{BD-3} to obtain the following estimate,
\begin{align}
\begin{aligned}\label{E8}
&\left| \theta^h_j(n_*)-\frac{\theta^h_P(n_*) + \theta^h_p(n_*)}{2}\right|\\
 &= \frac{1}{2} \left|-( \theta^h_P(n_*) - \theta^h_j(n_*)) + \theta^h_j(n_*) - \theta^h_p(n_*) \right| \\
 &\le \frac{1}{2} \max \left\{ | \theta^h_j(n_*) - \theta^h_P(n_*)|, |\theta^h_j(n_*) - \theta^h_p(n_*)|\right\} \le \frac{l}{2}.
\end{aligned}
\end{align}
Then, we apply \eqref{E8} and the simple triangle inequality to obtain the estimate of the last term in \eqref{E6} as follows,
\begin{align}
\begin{aligned}\label{BD-5}
&\left| \theta^h_j(n_*)-\frac{\theta^h_Q(n_*) + \theta^h_q(n_*)}{2}\right| \\
&\le \left| \theta^h_j(n_*)-\frac{\theta^h_P(n_*) + \theta^h_p(n_*)}{2}\right| + \left| \frac{\theta^h_P(n_*) + \theta^h_p(n_*)}{2} - \frac{\theta^h_Q(n_*) + \theta^h_q(n_*)}{2} \right| \\
&\le \frac{l}{2} + \frac{1}{2} \max \left\{ |\theta^h_P(n_*) - \theta^h_Q(n_*)| , |\theta^h_p(n_*) - \theta^h_q(n_*)|\right\}.
\end{aligned}
\end{align}
Next we will estimate $|\theta^h_P(n_*) - \theta^h_Q(n_*)|$ and $|\theta^h_p(n_*) - \theta^h_q(n_*)|,$ respectively. According to the definition \eqref{E4}, it is clear that 
\begin{equation}\label{BD-6}
\begin{aligned}
&\theta^h_Q(n_*) \le \theta^h_P(n_*), \quad \theta^h_q(n_*) \ge \theta^h_p(n_*), \\
&\theta^h_Q(n_*+1) \ge \theta^h_P(n_*+1), \quad \theta^h_q(n_*+1) \le \theta^h_p(n_*+1).
\end{aligned}
\end{equation}
Therefore, according to the iteration scheme \eqref{A-2} and the uniform bound of the trigonometric functions, the quantity $|\theta^h_Q(n_*+1) - \theta^h_P(n_*+1)|$ can be estimated as below,
\begin{align}
\begin{aligned}\label{E11}
&\theta^h_Q(n_*+1) - \theta^h_P(n_*+1) \\
&= (\theta^h_Q(n_*) - \theta^h_P(n_*)) + h(\Omega_Q - \Omega_P) + \frac{Kh}{N} \sum_{j=1}^N [\sin (\theta^h_j(n_*) - \theta^h_Q(n_*)) - \sin (\theta^h_j(n_*) - \theta^h_P(n_*))] \\
&\le - (\theta^h_P(n_*) - \theta^h_Q(n_*)) + hD(\Omega) + 2Kh.
\end{aligned}
\end{align}
Thus, we combine \eqref{BD-6} and \eqref{E11} to obtain the estimate of the quantity $|\theta^h_P(n_*) - \theta^h_Q(n_*)|$ as follows,
\begin{equation}\label{E12}
\begin{aligned}
|\theta^h_P(n_*) - \theta^h_Q(n_*)|&=  \theta^h_P(n_*) - \theta^h_Q(n_*)\\
&\le -(\theta^h_Q(n_*+1) - \theta^h_P(n_*+1)) + h(D(\Omega) + 2K)\leq h(D(\Omega) + 2K).
\end{aligned}
\end{equation}
Similarly, we can apply the same argument as above to obtain the estimate of the difference between the $p$-th and $q$-th oscillators, i.e. 
\begin{equation}\label{BD-8}
|\theta^h_q(n_*) - \theta^h_p(n_*)|  \le h(D(\Omega) + 2K).
\end{equation}

%On the other hand, 
%\begin{align*}
%\begin{aligned}
%\theta_p(n_*+1) - \theta_q(n_*+1) &= \theta^h_p(n_*) + h\Omega_m + \frac{Kh}{N} \sum_{j=1}^N \sin (\theta_j(n_*) - \theta^h_p(n_*))\\
%&\hspace{0.5cm} - \theta^h_q(n_*) - h\Omega_q - \frac{Kh}{N} \sum_{j=1}^N \sin (\theta_j(n_*) - \theta^h_q(n_*)) \\
%&=- (\theta^h_q(n_*) - \theta_p(n_*)) + h(\Omega_m - \Omega_q) \\
%&\hspace{0.5cm} + \frac{Kh}{N} \sum_{j=1}^N [\sin (\theta_j(n_*) - \theta_p(n_*)) - \sin (\theta_j(n_*) - \theta^h_q(n_*))] \\
%&\le - (\theta^h_q(n_*) - \theta_p(n_*)) + hD(\Omega) + 2Kh
%\end{aligned}
%\end{align*}
%Then we obtain that
%\[ \theta^h_q(n_*) - \theta_p(n_*) \le -(\theta_p(n_*+1) - \theta_q(n_*+1)) + h(D(\Omega) + 2K).\]
%It follows from \eqref{BD-6} that
%\begin{equation}\label{BD-8}
%|\theta^h_q(n_*) - \theta_p(n_*)| = \theta^h_q(n_*) - \theta_p(n_*) \le h(D(\Omega) + 2K).
%\end{equation}

\noindent Since $0 < l < 2 \arccos \frac{N- N_0}{N_0} \le \pi$, thus for sufficiently small $h$, we combine \eqref{BD-5}, \eqref{E12} and \eqref{BD-8} to obtain for $1 \le j \le N_0$ that,
\begin{equation}\label{h-1}
\left| \theta^h_j(n_*)-\frac{\theta^h_Q(n_*) + \theta^h_q(n_*)}{2}\right| \le \frac{l}{2} + \frac{D(\Omega) +2K}{2} h < \frac{\pi}{2}.
\end{equation}
Thus for $1 \le j \le N_0,$ it is easy to see that
\begin{equation}\label{BD-9}
\cos \left(\theta^h_j(n_*)-\frac{\theta^h_Q(n_*) + \theta^h_q(n_*)}{2} \right)\ge \cos \left(\frac{l}{2} + \frac{D(\Omega) +2K}{2} h \right) > 0.
\end{equation}
Next, to estimate the sinusoidal part in \eqref{BD-4}, we have to study the difference between the $Q$-th and $q$-th oscillators at step $n_*$. According to $\eqref{BD-3}_1$, \eqref{E4} and \eqref{BD-4}, we can let $h$ be sufficiently small to obtain,
\begin{align}
\begin{aligned}\label{h-2}
\pi>l>\theta^h_Q(n_*) - \theta^h_q(n_*) \ge l - hD(\Omega) - 2Kh > 0,
\end{aligned}
\end{align}
where we choose $h < \frac{l}{D(\Omega) +2K}.$ Thus we combine \eqref{BD-19} and \eqref{h-2} to obtain the positivity of the last term in \eqref{BD-4},
\begin{equation*}
\sin \frac{\theta^h_Q(n_*) - \theta^h_q(n_*)}{2} \sum_{j=1}^N \cos \left( \theta^h_j(n_*)-\frac{\theta^h_Q(n_*) + \theta^h_q(n_*)}{2}\right)  > 0.\newline
\end{equation*}

\noindent $\bullet$ Step 2. (Uniform bound) In this step, we will continue to prove the uniform bound of the first $N_0$ oscillators. According to the property of trigonometric functions and the estimate \eqref{h-1}, we obtain for small $h > 0$ that,
\begin{align}
\begin{aligned}\label{BD-10}
&\cos \left(\frac{l}{2} + \frac{D(\Omega) +2K}{2} h \right) \\
&= \cos \frac{l}{2} \cos \left(\frac{D(\Omega) + 2K}{2}h \right) - \sin \frac{l}{2} \sin (\frac{D(\Omega) + 2K}{2}h ) \\
&\ge \cos \frac{l}{2} \left[ 1 - \frac{(D(\Omega) + 2K)^2}{8}h^2\right] - \sin \left(\frac{D(\Omega) + 2K}{2} h \right) \\
&\ge \cos \frac{l}{2} - \cos \frac{l}{2} \frac{(D(\Omega) + 2K)^2}{8} h^2 - \frac{D(\Omega) + 2K}{2}h \\
&\ge \cos \frac{l}{2} - \left[ \cos \frac{l}{2} \frac{(D(\Omega) + 2K)^2}{8} + \frac{D(\Omega) + 2K}{2}\right]h.
\end{aligned}
\end{align}
%where we use the inequalities $\cos x \ge 1 - \frac{x^2}{2}$, $\cos x \ge -1$ and $\sin x < x, x \in (0, +\infty).$
Therefore, for sufficiently small $h$, we combine \eqref{h-1}, \eqref{BD-10} and the property of trigonometric functions to obtain that
%\begin{align*}
%\begin{aligned}
%&\sum_{j=1}^{N_0} \cos \left( \theta_j(n_*)-\frac{\theta^h_q(n_*) + \theta^h_q(n_*)}{2}\right) \\
%&\ge N_0 \left[ \cos \frac{l}{2} - \left( \cos \frac{l}{2} \frac{(D(\Omega) + 2K)^2}{4} + \frac{D(\Omega) + 2K}{2}\right)h\right]
%\end{aligned}
%\end{align*}
%It yields that
\begin{align}
\begin{aligned}\label{BD-11}
&\sum_{j=1}^N \cos \left( \theta^h_j(n_*)-\frac{\theta^h_Q(n_*) + \theta^h_q(n_*)}{2}\right) \\
&=\left(\sum_{j=1}^{N_0}+\sum_{j=N_0+1}^N\right) \cos \left( \theta^h_j(n_*)-\frac{\theta^h_Q(n_*) + \theta^h_q(n_*)}{2}\right) \\
&\ge N_0 \cos \frac{l}{2} - N_0 \left( \cos \frac{l}{2} \frac{(D(\Omega) + 2K)^2}{8} + \frac{D(\Omega) + 2K}{2}\right)h - (N- N_0).
\end{aligned}
\end{align}
According to \eqref{BD-1}, we have $N_0 \cos \frac{l}{2} - (N- N_0) > 0$. More over, for simplicity, we define the following notation,
\[A:= N_0 \left( \cos \frac{l}{2} \frac{(D(\Omega) + 2K)^2}{8} + \frac{D(\Omega) + 2K}{2}\right) >0.\]
Then, by choosing $h$ sufficiently small such that \eqref{BD-11} holds and, moreover, $h < \frac{N_0 \cos \frac{l}{2} - (N- N_0)}{A}$, we obtain
\begin{equation}\label{BD-12}
\sum_{j=1}^N \cos \left( \theta^h_j(n_*)-\frac{\theta^h_Q(n_*) + \theta^h_q(n_*)}{2}\right) \ge N_0 \cos \frac{l}{2} - (N- N_0) - Ah > 0.
\end{equation}
On the other hand, according to the iteration scheme \eqref{BD-4} and the positivity in \eqref{BD-9} and \eqref{h-2}, we obtain that
\[\theta^h_{Q}(n_* + 1) - \theta^h_q(n_* + 1)  \le \theta^h_{Q}(n_*) - \theta^h_q(n_*) + h D(\Omega).\]
Then, according to \eqref{BD-3} and the definition \eqref{E4}, we can choose $h$ sufficiently small such that $h < \frac{l}{D(\Omega)}$ to obtain that 
%Due to \eqref{BD-3} and the definition \eqref{E4},
%$\theta_{p}(n_*+1) -  \theta_q(n_*+1) \ge l,$
%\[l \le \theta_{N_0}(n_* + 1) - \theta_1(n_* + 1) \le \theta_{N_0}(n_*) - \theta_1(n_*) + h D(\Omega),\]
%by choosing $h < \frac{l}{D(\Omega)}$ we get that 
\begin{equation}\label{BD-13}
\left\{\begin{aligned}
&\pi>l>\theta^h_{Q}(n_*) - \theta^h_q(n_*) \ge l - hD(\Omega) > 0,\\
&\sin \frac{\theta^h_{Q}(n_*) - \theta^h_q(n_*)}{2} \ge \sin \frac{l - hD(\Omega)}{2} > 0.
\end{aligned}
\right.
\end{equation}
Then, we combine \eqref{BD-13} and the properties of trigonometric functions to obtain for sufficiently small $h$ that
\begin{align}
\begin{aligned}\label{E22}
\sin \frac{l - hD(\Omega)}{2}&= \sin \frac{l}{2} \cos \frac{hD(\Omega)}{2} - \cos \frac{l}{2} \sin \frac{hD(\Omega)}{2} \\
&\ge \sin \frac{l}{2} - \sin \frac{l}{2} \frac{(D(\Omega))^2}{8} h^2 - \frac{D(\Omega)}{2}h \ge \sin \frac{l}{2} - Bh,
\end{aligned}
\end{align}
where $B := \sin \frac{l}{2} \frac{(D(\Omega))^2}{8} + \frac{D(\Omega)}{2}  > 0$. Then based on the choices of small $h$, we may combine \eqref{BD-4}, \eqref{h-1}, \eqref{BD-12}, \eqref{BD-13} and \eqref{E22} to yield that
\begin{equation}\label{E23}
\begin{aligned}
&\theta^h_Q(n_*+1) - \theta^h_q(n_*+1)\\
 &\le \theta^h_Q(n_*) - \theta^h_q(n_*) + hD(\Omega) - \frac{2Kh}{N} \sin \frac{l - hD(\Omega)}{2} \left[N_0 \cos \frac{l}{2} - (N- N_0) - Ah\right]\\
 &\le \theta^h_Q(n_*) - \theta^h_q(n_*) + hD(\Omega) - \frac{2Kh}{N}  \left( \sin \frac{l}{2} - Bh\right) \left[N_0 \cos \frac{l}{2} - (N- N_0) - Ah\right] \\
&\le \theta^h_Q(n_*) - \theta^h_q(n_*) + hD(\Omega) - \frac{2Kh}{N} \sin \frac{l}{2} \left[N_0 \cos \frac{l}{2} - (N- N_0)\right] + Ch^2 + E h^2 \\
&\le \theta^h_Q(n_*) - \theta^h_q(n_*) + h [-F +( C + E)h],
\end{aligned}
\end{equation}
where for notational simplicity, we apply the notations $A$ in \eqref{BD-12} and $B$ in \eqref{E22} to further define $C$, $E$ and $F$ as below, 
\begin{align}
\left\{\begin{aligned}\label{E24}
&C := \frac{2KB}{N} \left[N_0 \cos \frac{l}{2} - (N- N_0)\right]  > 0, \quad E := \frac{2KA}{N} \sin \frac{l}{2} > 0,\\
&F :=  \frac{2K}{N} \sin \frac{l}{2} \left[N_0 \cos \frac{l}{2} - (N- N_0)\right] - D(\Omega)>0,
\end{aligned}
\right.
\end{align}
where $F>0$ is due to the assumption on $K$ in \eqref{BD-1}. Therefore, we combine the constraints on $h$ in \eqref{h-1}, \eqref{BD-12}, \eqref{BD-13} and further choose $h < \frac{F}{C+E}$ to set $h_0$ as below,
%Since
%\[2\sin \frac{l}{2}[N_0 \cos \frac{l}{2} - (N - N_0)]  = [N_0 \sin l - 2(N-N_0) \sin \frac{l}{2}],\]
%\[ K > \frac{D(\Omega)}{\frac{N_0}{N} \sin l - \frac{2(N-N_0)}{N} \sin \frac{l}{2}},\]
%\[\frac{2Kh}{N} \sin \frac{l}{2} [N_0 \cos \frac{l}{2} - (N-N_0)] =\frac{Kh}{N}[N_0 \sin l - 2(N-N_0) \sin \frac{l}{2}] > hD(\Omega). \]
\[h_0 := \min \left\{\frac{\pi - l}{D(\Omega) +2K},\quad \frac{N_0 \cos \frac{l}{2} - (N- N_0)}{A},\quad \frac{l}{D(\Omega)},\quad \frac{F}{C+E}\right\}.\]
Then for the mesh size $0 < h < h_0$, we combine \eqref{BD-3}, \eqref{E23} and \eqref{E24}  to obtain that
\[l \le \theta^h_Q(n_*+1) - \theta^h_q(n_*+1) \le \theta^h_Q(n_*) - \theta^h_q(n_*) < l,\]
which is obviously a contradiction. Therefore, we complete the proof of of the lemma.

\end{proof}
In the next lemma, we will make the first $N_0$ oscillators as a reference and imply the uniform boundedness of all oscillators. Thus, the gradient flow structure guarantees the convergence of the solution to the static state.

\begin{lemma}\label{2pi-vital}
Let $N\geq 3$, suppose that the initial configuration $\Theta(0)$ and natural frequencies satisfy the following conditions,
\[\frac{1}{N} \sum_{j=1}^N \theta_j(0) = 0, \quad \frac{1}{N} \sum_{j=1}^N \Omega_j = 0,,\quad \theta_j(0) \in [-\pi, \pi), \quad 1 \le j \le N. \]
Moreover, we let $N_0$, $l$ and $K$ be positive constants which satisfy the following conditions,
\begin{equation}
\begin{aligned}
&N_0 \in \bbz^+ \cap \left(\frac{N}{2}, N\right], \qquad l \in \left(0, 2 \arccos \frac{N-N_0}{N_0}\right), \\
& \max_{1 \le j,k \le N_0} |\theta_j(0) - \theta_k(0)| < l, \qquad K > \frac{D(\Omega)}{\frac{N_0}{N} \sin l - \frac{2(N-N_0)}{N} \sin \frac{l}{2}}. 
\end{aligned}
\end{equation}
Then, for the solution $\Theta^h(n)$  to the discrete non-identical system \eqref{A-2}, there exist a positive constant $h_0 $ and a equilibrium state $\Theta^\infty$ such that, for $0 < h < h_0$ we have
\[\sup_{0 \le n < +\infty} D(\Theta^h(n)) \le 4\pi + 2l, \quad \lim_{n \to +\infty} ||\Theta^h(n) - \Theta^\infty||_\infty = 0.\]
\end{lemma}
\begin{proof}
We will prove the lemma in two steps. In the first step, we will show the uniform boundedness of the phase diameter $D(\Theta^h(n))$. While in the second step, we will apply the gradient flow structure and Theorem \ref{T3.1} to prove the convergence of the phases of oscillators.\newline
 
\noindent $\bullet$ Step A. (Uniform bound of relative distance) In this step, we will study the dynamics of the oscillators $\{\theta^h_1(n), \ldots, \theta^h_{N}(n)\}$ and prove by contradiction that the relative distance between any two oscillators is uniformly bounded. Suppose not, i.e. 
\[\limsup_{n\rightarrow+\infty}D(\Theta^h(n))=+\infty.\]
In Lemma \ref{L5.1}, we already prove the uniform bound of the first $N_0$ oscillators. therefore we may define
\[S_0(n) :=\{\theta^h_1(n), \ldots, \theta^h_{N_0}(n)\}, \quad S_{-1} := S_0 - 2\pi, \quad S_1 := S_0 + 2\pi,\]
%According to Lemma \ref{L5.1}, we know that $S_0(n)$ is uniformly bounded. Then, we may assume that
%\[\theta_{N_0+1}(0), \ldots, \theta_N(0) \in [-\pi, \pi) \cap (S_0(0))^c.\]
Then, combining the zero mean phase property and the assumption $\limsup_{n\rightarrow+\infty}D(\Theta^h(n))=+\infty$, we immediately conclude that, at least one oscillator in the set $\{\theta^h_{N_0+1},\ldots, \theta^h_N\}$ is unbounded with respect to $S_0(n)$, say $\theta^h_{N_0 + 1}$. Then before it tends to $-\infty$ or $+\infty$, this oscillator will enter the neighborhood of one of the sets $S_{k}$, $k=-1,1$. Without loss of generality, we may assume $\theta^h_{N_0 + 1}$ gets into $S_1$ for some steps. In other words, we may assume that there exists a step $n_e$ such that
\[\max_{1 \le k \le N_0} |\theta^h_{N_0+1}(n_e) -  (\theta^h_k(n_e) + 2 \pi)| < l.\]
Then, we claim that the oscillator  $\theta^h_{N_0 + 1}$ will stay in the above region for all steps after $n_e$. More precisely, 
\begin{equation}\label{BD-15}
\max_{1 \le k \le N_0} |\theta^h_{N_0+1}(n) - (\theta^h_k(n) + 2\pi)| < l, \quad \text{for all } n \ge n_e.
\end{equation}
Suppose \eqref{BD-15} does not hold, then there exists a step $n_*$, satisfying $n_* \ge n_e$, such that
\begin{equation}\label{E27}
\left\{\begin{aligned}
&\max_{1 \le j \le N_0} |\theta^h_{N_0+1}(n) - (\theta_j^h(n) + 2\pi)| < l, \quad n_e \le n \le n_*,\\
&\max_{1 \le j \le N_0} |\theta^h_{N_0+1}(n_*+1) - (\theta^h_j(n_*+1) + 2\pi)| \ge l.
\end{aligned}
\right.
\end{equation}
According to \eqref{E27} and Lemma \ref{L5.1}, there exist two possibilities at $(n_*+1)$-th step. More precisely, $\theta^h_{N_0+1}(n_*+1)$ is either greater than $\left(\min_{1 \le j \le N_0} \theta^h_j(n_*+1) +2\pi+l \right)$, or less than $\left(\max_{1 \le j \le N_0} \theta^h_j(n_*+1) -2\pi -  l \right) $. As the two cases can be dealt with the same manner, in the following, we will only study the case below,
\begin{equation}\label{BD-16}
\theta^h_{N_0+1}(n_*+1) - \left(\min_{1 \le j \le N_0} \theta^h_j(n_*+1) +2\pi \right) \ge l.
\end{equation}
Let's assume the $q$-th oscillator to be the minimum in the first $N_0$ oscillators at step $n_*+1$. Then, we have the estimate of \eqref{BD-16} as below,
\begin{align}
\begin{aligned}\label{BD-18}
&\theta^h_{N_0+1}(n_*+1) - (\theta^h_q(n_*+1) + 2\pi) \\
&= \theta^h_{N_0+1}(n_*) - (\theta^h_q(n_*) + 2\pi) + h(\Omega_{N_0+1} - \Omega_q) \\
&+ \frac{Kh}{N} \sum_{j=1}^N [\sin (\theta^h_j(n_*) - \theta^h_{N_0+1}(n_*)) - \sin (\theta^h_j(n_*) - (\theta^h_q(n_*)+2\pi))] \\
&\le \theta^h_{N_0+1}(n_*) - (\theta^h_q(n_*) + 2\pi) + hD(\Omega) \\
&- \frac{2Kh}{N}  \left( \sin \frac{\theta^h_{N_0+1}(n_*) - (\theta^h_q(n_*)+2\pi)}{2} \right) \sum_{j=1}^N \cos \left( \theta^h_j(n_*) - \frac{\theta^h_{N_0+1}(n_*) + (\theta^h_q(n_*)+2\pi)}{2}\right),
\end{aligned}
\end{align}
Similar as in Lemma \ref{L5.1}, we will show the positivity of the last term. Actually, according to the iteration scheme \eqref{A-2} and \eqref{BD-16}, we have
\begin{align}
\begin{aligned}\label{E30}
&\theta^h_{N_0+1}(n_*) - (\theta^h_q(n_*) + 2\pi) \\
&= \theta^h_{N_0+1}(n_*+1) - (\theta^h_q(n_*+1) + 2\pi) - h(\Omega_{N_0+1} - \Omega_q) \\
& - \frac{Kh}{N} \sum_{j=1}^N [\sin (\theta^h_j(n_*) - \theta^h_{N_0+1}(n_*)) - \sin (\theta^h_j(n_*) - (\theta^h_q(n_*)+2\pi))] \\
&\ge l - hD(\Omega) - 2Kh.
\end{aligned}
\end{align}
Then, we combine \eqref{E27} and \eqref{E30} and choose a sufficiently $h$ such that $h < \frac{l}{D(\Omega) + 2K}$ to obtain the following estimates,
\begin{equation}\label{E31}
0 < \theta^h_{N_0+1}(n_*) - (\theta^h_q(n_*) + 2\pi) < l<\pi,\qquad \sin \frac{\theta^h_{N_0+1}(n_*) - (\theta^h_q(n_*)+2\pi)}{2} > 0.
\end{equation}
Next, we will study the cosine part in \eqref{BD-18}. Note $\theta^h_q(n_*)$ is not the minimum in the first $N_0$ oscillators at the step $n_*$. Therefore, we may assume the $p$-th oscillator to be minimum of the first $N_0$ oscillators, i.e.
\begin{equation}\label{E32}
\theta^h_p(n_*) + 2\pi = \min_{1 \le j \le N_0} \theta^h_j(n_*) + 2\pi.
\end{equation}

%Suppose $m \ne q.$ Note that
%\[\theta_q(n_*) + 2\pi \ge \theta_p(n_*) + 2\pi.\]
%Without loss of generality, it is assumed that
%\[  \theta_{N_0+1}(n_*) \le \max\limits_{1 \le j \le N_0} \theta_j(n_*) + 2\pi .\]

%Then there exists $M \ (1 \le M \le N_0)$ such that
%\[ \theta_{N_0+1}(n_*) \le \max\limits_{1 \le j \le N_0} \theta_j(n_*) + 2\pi =: \theta_P(n_*) + 2\pi .\]

\noindent Then, we can split the cosine part in \eqref{BD-18} into two parts and apply the boundedness property of trigonometric functions to obtain that 
\begin{align}
\begin{aligned}\label{BD-19}
&\sum_{j=1}^N \cos \left( \theta^h_j(n_*) - \frac{\theta^h_{N_0+1}(n_*) +(\theta^h_q(n_*)+2\pi)}{2}\right) \\
&= \left(\sum_{j=1}^{N_0}+\sum_{j=N_0+1}^N \right)\cos \left( (\theta^h_j(n_*) +2\pi)- \frac{\theta^h_{N_0+1}(n_*) + (\theta^h_q(n_*)+2\pi)}{2}\right) \\
&\geq  \sum_{j=1}^{N_0} \cos \left( \theta^h_j(n_*)+2\pi - \frac{\theta^h_{N_0+1}(n_*) + (\theta^h_q(n_*)+2\pi)}{2}\right)-(N-N_0).
\end{aligned}
\end{align}
Then the summation in \eqref{BD-19} can be estimated term by term. In fact, for $j$-th oscillator where $1\le j \le N_0$, we have  
\begin{align}
\begin{aligned}\label{BD-20}
&\left| (\theta^h_j(n_*) +2\pi)- \frac{\theta^h_{N_0+1}(n_*) + (\theta^h_q(n_*)+2\pi)}{2}\right| \\
&\leq \left| (\theta^h_j(n_*)+2\pi) - \frac{(\theta^h_{p}(n_*) +l +2\pi)+ (\theta^h_p(n_*)+2\pi)}{2} \right|\\
& + \left|\frac{(\theta^h_{p}(n_*)+l +2\pi)+ (\theta^h_p(n_*)+2\pi)}{2} - \frac{\theta^h_{N_0+1}(n_*) + (\theta^h_q(n_*)+2\pi)}{2}\right| \\
&=\mathcal{I}_1+\mathcal{I}_2.
\end{aligned}
\end{align}
For $\mathcal{I}_1$, as the diameter of the first $N_0$ oscillators are uniformly bounded by $l$, we apply $\eqref{E27}_1$, \eqref{E31} and \eqref{E32} to obtain that 
\begin{align}
\begin{aligned}\label{E35}
&\left| (\theta^h_j(n_*)+2\pi) - \frac{(\theta^h_{p}(n_*) +l +2\pi)+ (\theta^h_p(n_*)+2\pi)}{2}\right| \\
&= \frac{1}{2} \left| (\theta^h_j(n_*)+2\pi) - (\theta^h_{p}(n_*) +l+2\pi) + (\theta^h_j(n_*)+2\pi) - (\theta^h_p(n_*)+2\pi)\right| \\
&\le \frac{1}{2} \max \left\{ |(\theta^h_j(n_*)+2\pi) - (\theta^h_{p}(n_*) +l+2\pi)|, |(\theta^h_j(n_*)+2\pi) - (\theta^h_p(n_*)+2\pi)| \right\} \\
&\le \frac{l}{2}.
\end{aligned}
\end{align}
For $\mathcal{I}_2$, we first deal with the difference between $\theta^h_p(n_*)$ and $\theta^h_q(n_*)$. According to the definition of $p$-th and $q$-th oscillators in \eqref{BD-16} and \eqref{E32}, we have 
\begin{equation}\label{E36}
(\theta^h_p(n_*) + 2\pi) - (\theta^h_q(n_*) + 2\pi) \le 0, \quad (\theta^h_p(n_*+1) + 2\pi) - (\theta^h_q(n_*+1) + 2\pi) \ge 0.
\end{equation}
Therefore, according to the iteration scheme \eqref{A-2} and \eqref{E36}, we have the following estimates for the difference,
\begin{align}
\begin{aligned}\label{BD-22}
&(\theta^h_q(n_*) + 2\pi) - (\theta^h_p(n_*) + 2\pi) \\
&= -\left[ (\theta^h_p(n_*+1) +2\pi) - (\theta^h_q(n_*+1) + 2\pi)\right] + h(\Omega_p -\Omega_q) \\
& + \frac{Kh}{N} \sum_{j=1}^N [\sin (\theta^h_j(n_*) - \theta^h_p(n_*)) - \sin (\theta^h_j(n_*) - \theta^h_q(n_*))] \\
&\le hD(\Omega) + 2Kh.
\end{aligned}
\end{align}
Then we will deal with the difference between $\theta^h_{p}(n_*)+l +2\pi$ and $\theta^h_{N_0+1}(n_*) $ in $\mathcal{I}_2$. Similar as the estimate in \eqref{BD-22}, we can estimate the difference between $p$-th and $q$-th oscillator at step $n_*+1$ as below,
\begin{align}
\begin{aligned}\label{E38}
&(\theta^h_p(n_*+1) +2\pi) - (\theta^h_q(n_*+1) + 2\pi) \\
&= - \left[ (\theta^h_q(n_*) + 2\pi) - (\theta^h_p(n_*) + 2\pi)\right] + h(\Omega_p -\Omega_q) \\
& + \frac{Kh}{N} \sum_{j=1}^N [\sin (\theta^h_j(n_*) - \theta^h_p(n_*)) - \sin (\theta^h_j(n_*) - \theta^h_q(n_*))]\\
&\leq hD(\Omega) + 2Kh.
\end{aligned}
\end{align}
Then, we directly apply the iteration scheme \eqref{A-2}, \eqref{BD-16} and \eqref{E38} to obtain the estimate of the difference between $\theta^h_{p}(n_*)+l +2\pi$ and $\theta^h_{N_0+1}(n_*) $ as below,
\begin{align}
\begin{aligned}\label{BD-21}
&(\theta^h_p(n_*)+l + 2\pi) - \theta^h_{N_0+1}(n_*) \\
&= \left[  (\theta^h_q(n_*+1) +l+2\pi)-\theta^h_{N_0+1}(n_*+1)\right]+[\theta^h_p(n_*+1)-\theta^h_q(n_*+1)] \\
&+ h(\Omega_{N_0+1} - \Omega_p)  + \frac{Kh}{N} \sum_{j=1}^N [\sin (\theta^h_j(n_*) -\theta^h_{N_0+1}(n_*)) - \sin (\theta^h_j(n_*) - \theta^h_p(n_*))] \\
&\le 2hD(\Omega) + 4Kh.
\end{aligned}
\end{align}
Therefore, by choosing $h$ sufficiently small, we combine the estimates \eqref{BD-20}, \eqref{E35}, \eqref{BD-22}, and \eqref{BD-21} to obtain for $1 \le j \le N_0$ that
\begin{equation}\label{E40}
\left| (\theta^h_j(n_*) +2\pi)- \frac{\theta^h_{N_0+1}(n_*) + (\theta^h_q(n_*)+2\pi)}{2}\right| \le \frac{l}{2} + \frac{3D(\Omega) + 6K}{2}h < \frac{\pi}{2},
\end{equation}
Combining \eqref{E31} and \eqref{E40}, we obtain that the last term in \eqref{BD-18} is positive. Therefore, we can apply the same argument in the second step of the proof of Lemma \ref{L5.1} to show that there exists a positive constant $h_0$ such that 
%where we choose small $h < \frac{\pi - l}{D(\Omega) +2K}$ such that
%\[ l + [D(\Omega) +2K]h < \pi.\]
\[\theta^h_{N_0+1}(n_*+1) - (\theta^h_q(n_*+1) + 2\pi)<l,\qquad 0<h<h_0,\]
which is obviously a contradiction to \eqref{BD-16}. Thus the proof of claim \eqref{BD-15} is completed. Moreover, the case when oscillator $\theta^h_{N_0+1}$ enter the set $S_{-1}$ at step $n_e$ can be dealt with the same method. Therefore, we combine Lemma \ref{L4.1}, \eqref{BD-15} to obtain that 
\begin{equation}\label{E41}
\max_{1\leq j\leq N_0}\theta^h_j(n) - 2\pi-l<\theta^h_{N_0+1}(n) < \min_{1\leq j\leq N_0}\theta^h_j(n) + 2\pi+l,\quad n\geq n_e.
\end{equation}
\eqref{E41} immediately implies the uniform bound of the diameter $D(\Theta^h(n))$, which is a contradiction to the assumption $\limsup_{n\rightarrow+\infty}D(\Theta^h(n))=+\infty$. Therefore, we conclude that the diameter $D(\Theta^h(n))$ is uniformly bounded. Moreover, according to \eqref{E41}, we have  
\begin{equation}\label{BD-23}
\sup_{0 \le n < +\infty} D(\Theta^h(n)) \le 4\pi + 2l.
\end{equation}
\vspace{0.5cm}

\noindent $\bullet$ Step 2. (Asymptotical synchronization)~According to the conservation of the total phase, we have the zero total phase for any solution to \eqref{A-2} with initial data stated in the lemma. This directly impies that
\[|\theta^h_k(n)| \le D(\Theta^h(n)) \le \sup_{0 \le n \le +\infty} D(\Theta^h(n))<4 \pi+2l.\]
Then we exploit Theorem \ref{T3.1} to obtain that for sufficiently small $0 < h < h_0$, there exists $\Theta^\infty$ such that
\[\lim_{n \to +\infty} || \Theta^h(n) - \Theta^\infty||_\infty = 0.\]
\end{proof}
Now, we are ready to prove the main theorem in the present section. For any given initial data, we only need to check if the conditions in Lemma \ref{2pi-vital} hold at some step. As in \cite{H-K-R}, the authors already verified this in the continuous model, we can simply use the approximation between continuous model and discrete model to obtain the desired results. 

\begin{theorem}\label{final}
Let $N\geq 3$, suppose that the initial configuration $\Theta_0$ and natural frequencies $\Omega_i$ satisfy the conditions \eqref{multiple}
%\[\frac{1}{N} \sum_{j=1}^N \Omega_j = 0, \quad \frac{1}{N} \sum_{j=1}^N \theta_j^0 = 0, \quad \theta_j^0 \in [-\pi, \pi), \ 1 \le j \le N,\]
and
\[r_0 >0, \quad \theta_j^0 \ne \theta_k^0, \ 1 \le j \ne k \le N, \quad ||\Omega||_\infty = \max_{1 \le j \le N } |\Omega_j| < \infty.\]
Then, there exists a large coupling strength $K_\infty > 0$ and a small mesh size $h_0 >  0$ such that,  if $K> K_\infty$ and $0 < h < h_0$, then the emergence of phase-locked state will asymptotically occurs. More precisely, we can find a phase locked state $\Theta^\infty$ such that the solution to system \eqref{A-2} with initial data $\Theta_0$ satisfies
\[\lim_{n \to \infty} ||\Theta^h(n) - \Theta^\infty||_\infty = 0,\] 
provided that $K> K_\infty$ and $0 < h < h_0$. 
\end{theorem}
\begin{proof}
According to \cite{H-K-R}, for continuous model and the same initial condition, we can find a sufficient large $K_\infty$. If the coupling strength $K\geq K_\infty$, then there exists a time $T_\varepsilon$ such that the conditions in Lemma \ref{2pi-vital} hold at $T_\varepsilon$.
\begin{equation}\label{E43}
\begin{aligned}
&N_0 =N-1, \qquad l \in \left(0, 2 \arccos \frac{N-N_0}{N_0}\right), \\
& \max_{1 \le j,k \le N_0} |\theta_j(T_\varepsilon) - \theta_k(T_\varepsilon)| < l, \qquad K > \frac{D(\Omega)}{\frac{N_0}{N} \sin l - \frac{2(N-N_0)}{N} \sin \frac{l}{2}}. 
\end{aligned}
\end{equation}
Then, we apply the continuity of the solution $\Theta(t)$ and the approximation property in Lemma \ref{L2.5} to conclude for sufficiently small $h$ that, there exists a step $n_{\varepsilon}$ such that
\begin{equation}\label{E44}
\begin{aligned}
&N_0 =N-1, \qquad l \in \left(0, 2 \arccos \frac{N-N_0}{N_0}\right),\qquad K > \frac{D(\Omega)}{\frac{N_0}{N} \sin l - \frac{2(N-N_0)}{N} \sin \frac{l}{2}}, \\
& \max_{1 \le j,k \le N_0} |\theta^h_j(n_\varepsilon) - \theta^h_k(n_\varepsilon)| < l,\qquad n_\varepsilon h\leq T_\varepsilon\leq (n_\varepsilon +1)h. 
\end{aligned}
\end{equation}
Now, we combine Lemma \ref{2pi-vital} and \eqref{E44} to finish the proof of the main theorem. 

\end{proof}

\section{Summary}\label{sec:6}
\setcounter{equation}{0}
\vspace{0.5cm}
In this paper, we first provided a discrete version of the gradient flow theory, and accordingly prove the emergence of synchronization of the discrete-time Kuramoto model in both identical and non-identical cases. Then, in order to yield the convergence rate, we apply the approximation between continuous model and discrete model to obtain the exponential decay rate for discrete identical Kuramoto model. Moreover, according to the definition \eqref{D1}, for initial data $\Theta^h_0\in\mathcal{A}_1$, we proved that the time asymptotical equilibrium states of discrete and continuous models coincide with each other. While for initial data $\Theta^h_0\in\mathcal{A}_2$, we cannot prove this, which may be due to the instability of the bipolar states. Finally,  for non-identical model, we apply the theory of discrete gradient flow to yield the emergence of synchronization. However, we can only obtain the exponential decay for initial data in $\mathcal{A}_1$ so far. Therefore, we can apply the results in \cite{H-K-K-Z} to conclude the uniform-in-time convergence from discrete Kuramoto model to continuous Kuramoto model for $\Theta^h_0\in\mathcal{A}_1$. The case $\Theta^h_0\in\mathcal{A}_2$ will be studied in our future work.\newline

\end{document}